\documentclass[reqno,10pt,a4paper]{amsart}
\usepackage[margin=2.7cm]{geometry}
\usepackage{cite}
\usepackage{newpxtext}
\usepackage{cabin}
\usepackage[varqu,varl]{inconsolata}
\usepackage[bigdelims,vvarbb]{newpxmath}
\usepackage[cal=cm,bb=esstix,bbscaled=1.05]{mathalfa}
\usepackage[labelformat=simple]{subcaption}
\usepackage{tikz}
\usepackage{bm}
\usepackage{mathtools}
\usepackage{pgfplots}
\pgfplotsset{compat=1.15}
\usetikzlibrary{arrows}
\usepackage{xcolor}
\usepackage{hyperref}
\hypersetup{
 colorlinks,
 linkcolor={red!50!black},
 citecolor={blue!50!black},
 urlcolor={blue!80!black}
}

\numberwithin{equation}{section}
\AtBeginDocument{%
\def\MR#1{}
}
\newtheorem{theorem}{Theorem}[section]
\newtheorem{lemma}[theorem]{Lemma}
\newtheorem{proposition}[theorem]{Proposition}
\newtheorem{corollary}[theorem]{Corollary}
\newtheorem{conjecture}[theorem]{Conjecture}
\theoremstyle{definition}
\newtheorem{definition}[theorem]{Definition}
\newtheorem{example}[theorem]{Example}
\newtheorem{remark}[theorem]{Remark}

\renewcommand{\vec}[1]{\bm{#1}}
\newcommand{\ed}[1]{e^*_{#1}}
\DeclareMathOperator{\res}{res}
\newcommand{\Ec}{\mathcal{E}}
\newcommand{\PP}{\mathbb{P}}
\newcommand{\ZZ}{\mathbb{Z}}
\newcommand{\QQ}{\mathbb{Q}}
\newcommand{\NQ}{N_{\QQ}}
\newcommand{\MQ}{M_{\QQ}}
\newcommand{\set}[1]{\left\{ #1 \right\}}
\newcommand{\setcond}[2]{\set{#1 \ \colon \ #2}}
\newcommand{\verts}{\mathcal{V}}
\newcommand{\Hc}{\mathcal{H}}
\newcommand{\Bc}{\mathcal{B}}
\newcommand{\Fc}{\mathcal{F}}
\newcommand{\dirbasket}{\overrightarrow{\Bc}}
\newcommand{\Rc}{\mathcal{R}}
\newcommand{\boundary}{\partial}
\newcommand{\Cc}{\mathcal{C}}
\newcommand{\hmin}{h_{\min}}
\newcommand{\hmax}{h_{\max}}
\newcommand{\elltilde}{\widetilde{\ell}}
\DeclareMathOperator{\conv}{conv}
\DeclareMathOperator{\cone}{cone}
\DeclareMathOperator{\Hom}{Hom}
\DeclareMathOperator{\strint}{int}
\DeclareMathOperator{\GL}{GL}
\DeclareMathOperator{\SL}{SL}
\DeclareMathOperator{\mut}{mut}
\DeclareMathOperator{\Vol}{Vol}
\DeclareMathOperator{\Ehr}{Ehr}
\DeclareMathOperator{\SC}{SC}
\DeclareMathOperator{\Aut}{Aut}
\newcommand{\isom}{\cong}
\newcommand{\congr}{\equiv}
\newcommand{\KE}{K{\"a}hler\nobreakdash--Einstein}
\newcommand{\doi}[1]{\href{https://doi.org/#1}{doi:#1}}

\begin{document}
\author[T.\,Hall]{Thomas~Hall}
\address{School of Mathematical Sciences\\University of Nottingham\\Nottingham\\NG7 2RD\\UK}
\email{pmyth1@nottingham.ac.uk}
\makeatletter
\@namedef{subjclassname@2020}{%
 \textup{2020} Mathematics Subject Classification}
\makeatother
\title{On the Uniqueness of K{\"a}hler--Einstein Polygons in Mutation-Equivalence Classes}
\begin{abstract}
We study a subclass of \KE\ Fano polygons and how they behave under mutation. The polygons of interest are \KE\ Fano triangles and symmetric Fano polygons. In particular, we find an explicit bound for the number of these polygons in an arbitrary mutation-equivalence class.
 
An important mutation-invariant of a Fano polygon is its singularity content. We extend the notion of singularity content and prove that it is still a mutation-invariant. We use this to show that if two symmetric Fano polygons are mutation-equivalent, then they are isomorphic. We further show that if two \KE\ triangles are mutation-equivalent, then they are isomorphic. Finally, we show that if a symmetric Fano polygon is mutation-equivalent to a \KE\ triangle, then they are isomorphic. Thus, each mutation-equivalence class has at most one Fano polygon which is either a \KE\ triangle or symmetric.

A recent conjecture states that all \KE\ Fano polygons are either triangles or are symmetric. We provide a counterexample~$P$ to this conjecture and discuss several of its properties. For instance, we compute iterated barycentric transformations of~$P$ and find that (a)~the \KE\ property is not preserved by the barycentric transformation, and (b)~$P$ is of strict type~$B_2$. Finally, we find examples of \KE\ Fano polygons which are not minimal.
\end{abstract}
\maketitle
\section{Introduction}
A new approach to the classification of Fano manifolds was recently initiated by Coates--Corti--Galkin--Golyshev--Kasprzyk~\cite{mirror_symmetric_fano_manifolds}. This has become known as the~\emph{Fanosearch programme}. Following this programme, a Laurent polynomial $f$ is said to be \emph{mirror dual} to a Fano manifold~$X$ if the classical period~$\pi_f$ of~$f$ agrees with the regularized quantum period~$\widehat{G}_X$ of~$X$; see~\cite{mirror_symmetric_fano_manifolds} for details. This correspondence is not unique: a Fano manifold can have (infinitely) many different mirror dual Laurent polynomials, and these Laurent polynomials are expected to be related via a combinatorial process called~\emph{mutation}~\cite{minkowski_mutation}. It is conjectured that Fano manifolds, up to~$\QQ$-Gorenstein\nobreakdash-deformation, are in bijective correspondence with certain Laurent polynomials, up to mutation.

Crucial to this approach is the Newton polytope~$P$ of the Laurent polynomial~$f$. This polytope~$P$ is~\emph{Fano}, i.e.\ it is a convex lattice polytope containing the origin in its strict interior, and has primitive vertices~\cite{FanoPoly}. We thus enter the world of toric geometry. The polytope~$P$ corresponds to a (possibily singular) toric Fano variety~$X_P$, which is expected to be a toric degeneration of the original Fano manifold~$X$. The notion of mutation can be extended from Laurent polynomials to Fano polytopes. Ilten showed that if two Fano polytopes~$P$ and~$Q$ are related by mutation, then the corresponding toric varieties~$X_P$ and~$X_Q$ are deformation-equivalent~\cite{mutation_deformations}.

In~\cite{mirror_symmetry_orbifolds}, the Fanosearch programme was specialised to orbifold del~Pezzo surfaces; that is, del~Pezzo surfaces with at worst cyclic quotient singularities. There are infinitely many of these surfaces, even up to $\QQ$-Gorenstein\nobreakdash-deformation; however, bounding the possible basket of singularities results in a finite classification. For example, an empty basket recovers the classical~$10$ smooth del Pezzo surfaces, up to $\QQ$-Gorenstein\nobreakdash-deformation.

Kasprzyk--Nill--Prince~\cite{minimality} introduced the notion of~\emph{minimality} for Fano polygons. A Fano polygon is called~\emph{minimal} if, amongst all polygons related to it by a single mutation, its volume is minimal. That is, minimality is a local property, and there can be multiple minimal polygons in a single mutation-equivalence class. Using this purely combinatorial definition, they found exactly~$26$ mutation-equivalence classes of Fano polygons whose baskets contain only~$\frac13(1,1)$ singularities; these agree with~$26$ of the del~Pezzo surfaces with~$\frac13(1,1)$ singularities classified by Corti--Heuberger~\cite{13sings}.

One of the main results of~\cite{minimality} is the following:
\begin{theorem}[\!{\cite[Theorem~6.3]{minimality}}]
There are finitely many minimal Fano polygons, up to isomorphism, with a given basket of singularities.
\end{theorem}
\noindent
As a consequence, the problem of classifying Fano polygons (up to mutation) with a given basket can be reduced to classifying finitely many minimal Fano polygons; this has been implemented algorithmically by Cavey--Kutas~\cite{CaveyKutas}. In this paper, we seek to further understand the behaviour of some of these minimal Fano polygons.

Another concept in algebraic geometry which has seen much recent interest is K-stability. It was initially formulated by Tian~\cite{kstab_tian} as a way to characterise the existence of \KE\ metrics on Fano manifolds. K-stability has since been expanded by Donaldson~\cite{kstab_donaldson} to polarised varieties. The Yau--Tian--Donaldson conjecture predicts that a Fano variety~$X$ is K-polystable if and only if it admits a \KE\ metric. This was proven in the smooth case by Chen--Donaldson--Sun~\cite{YTD_conj1, YTD_conj2} and Tian~\cite{YTD_conj3}; the singular case has seen recent progress~\cite{YTD_recent1, YTD_recent2}, but remains open. Its main application is the construction of well-behaved moduli spaces, called K-moduli spaces, of Fano varieties~\cite{K-moduli1, K-moduli2, K-moduli3}.

We study toric \KE\ Fano varieties. Following Hwang--Kim~\cite{barycentric_transformations}, a Fano polygon~$P$ is called~\emph{\KE} if its corresponding toric variety~$X_P$ admits a \KE\ metric. A related class, introduced by Batyrev--Selivanova~\cite{bat_sym}, are~\emph{symmetric polygons}: these are polygons whose automorphism group fixes only the origin.
\begin{theorem}[\!{\cite[Theorem~1.1]{bat_sym},~\cite[Theorem~1.3]{on_KEs}}]\label{thm:syms_are_KE}
All symmetric polytopes are \KE.
\end{theorem}
The converse to Theorem~\ref{thm:syms_are_KE} holds for smooth Fano polytopes in dimension less than seven. Nill--Paffenholz~\cite{nill_nonsym_KE} found that exactly one smooth Fano polytope in dimension seven, and exactly two smooth Fano polytopes in dimension eight, are \KE\ but not symmetric.

In this paper, we study symmetric Fano polygons and \KE\ Fano triangles. These were conjectured by Hwang--Kim~\cite{on_KEs} to constitute all \KE\ Fano polygons.
\begin{conjecture}[\!{\cite[Conjecture~1.6]{on_KEs}}]\label{conj:hk}
All \KE\ polygons are either symmetric or triangles.
\end{conjecture}
We prove that symmetric Fano polygons and \KE\ Fano triangles are minimal (see, respectively, Lemmas~\ref{lemma:P_minimal} and~\ref{lemma:KE_triangle_minimal}). So, in each mutation-equivalence class, there is a finite number of these polygons. The main result of this paper is the following:
\begin{theorem}\label{thm:main}
In each mutation-equivalence class, there is at most one Fano polygon which is either symmetric or a \KE\ triangle.
\end{theorem}
\noindent
Both behaviours are exhibited: there are examples of mutation-equivalence classes with no symmetric Fano polygons or \KE\ triangles (Example~\ref{eg:no_KE}) and examples with exactly one (Example~\ref{eg:one_KE}).

In~\S\ref{sec:proof_of_countereg} we provide a counterexample to Conjecture~\ref{conj:hk}:
\begin{proposition}\label{prop:countereg}
The Fano polygon
\[
P \coloneqq \conv\set{(-9, -190), (19,27), (15,113), (-13,112)} \subset \NQ
\]
is a \KE\ Fano quadrilateral. In particular, Conjecture~\ref{conj:hk} does not hold.
\end{proposition}
\noindent
All Fano quadrilaterals with barycentre zero are described in Proposition~\ref{prop:KE_quad_weights}. This allows us to present a method for systematically producing non-symmetric \KE\ quadrilaterals; see Remark~\ref{remark:KE_quad}.

As noted above, both symmetric Fano polygons and \KE\ Fano triangles are minimal. Furthermore, all the non-symmetric \KE\ Fano quadrilateral examples we have found are minimal. It is natural to ask whether this holds more generally for all \KE\ Fano polygons, however this is not the case:
\begin{proposition}[see Proposition~\ref{prop:KE_nonsym_from_3sym}]\label{prop:non_minimal_KE}
There exist \KE\ Fano polygons which are not minimal.
\end{proposition}
We now summarise the structure of this paper. In~\S\ref{section:prelims}, we fix our notation for the objects of study: polygons and mutation. We then extend the notion of basket of singularities to suit our needs. The next two sections are dedicated to proving Theorem~\ref{thm:main}. In~\S\ref{section:most_one_cs_fano}, we prove Theorem~\ref{thm:sym_unique}, which states that there is at most one symmetric Fano polygon in each mutation-equivalence class. In~\S\ref{section:KE-triangles}, we prove the rest of Theorem~\ref{thm:main}. In particular, we prove that there is at most one \KE\ Fano triangle in each mutation-equivalence class and that if a symmetric Fano polygon is mutation-equivalent to a \KE\ triangle, then they are isomorphic. We conclude with~\S\ref{section:countereg}, where we provide a counterexample to Conjecture~\ref{conj:hk}. We compute iterated barycentric transformations of~$P$, and show that it has strict type~$B_2$. Finally, we describe an alternative method for constructing non-symmetric \KE\ Fano polygons and use it to show that there are \KE\ Fano polygons which are not minimal (Proposition~\ref{prop:KE_nonsym_from_3sym}).
\section{Preliminaries}\label{section:prelims}
\subsection{Lattice polygons and basic invariants}
Throughout, we refer to the lattice~$N \isom \ZZ^2$ and its dual lattice~$M \coloneqq \Hom(N, \ZZ)$. We also refer to their~$\QQ$-extensions as~$\NQ \coloneqq N \otimes \QQ$ and~$\MQ \coloneqq M \otimes \QQ$. For elements~$u \in \MQ$ and~$\vec{x} \in \NQ$, we denote their natural pairing as~$u(\vec{x})$.

We often consider polygons up to~$\GL_2(\ZZ)$-equivalence. Let~$P, Q \subset \NQ$ be lattice polytopes. If there exists some~$U \in \GL_2(\ZZ)$ such that~$Q = UP$, then we say that~$P$ is~\emph{isomorphic} to~$Q$. A subtle point used throughout this paper is to also consider polygons only up to~$\SL_2(\ZZ)$-equivalence, since several properties we consider are only invariant under~$\SL_2(\ZZ)$-transformations and not~$\GL_2(\ZZ)$-transformations.

We first recall the definition of a Fano polygon. Each Fano polygon~$P \subset \NQ$ (up to~$\GL_2(\ZZ)$-equivalence) corresponds one-to-one with a normal toric Fano variety~$X_P$ (up to isomorphism).
\begin{definition}
Let~$P \subset \NQ$ be a lattice polygon, i.e.\ a polygon whose vertices are in the lattice~$N$. It is called~\emph{Fano} if the origin~$\vec{0}$ is in the strict interior~$\strint(P)$ of~$P$ and all its vertices~$\vec{v} = (x,y)$ are~\emph{primitive}, i.e.\ $\gcd(x,y) = 1$.
\end{definition}
It will be important later, when considering mutations, to have the notion of length and height of an edge of a Fano polygon.
\begin{definition}
Let~$E$ be an edge of a Fano polygon and let~$u \in M$ be the unique primitive inner normal vector to~$E$. Then we say that~$E$ has~\emph{(lattice) length}~$\ell_E \coloneqq |E \cap N| - 1$ and~\emph{height}~$h_E \coloneqq |u(\vec{v})|$, where~$\vec{v}$ is a vertex of~$E$. Note that the length and the height will be strictly positive integers.
\end{definition}
A classical way to describe a cone or an edge of a Fano polygon is by its singularity type~$\frac{1}{r}(1,a)$, which is equivalently a matrix in Hermite normal form~$\left(\begin{smallmatrix}
 1 & -a \\
 0 & r
\end{smallmatrix}\right)$.
It is invariant under~$\GL_2(\ZZ)$-transformations. So, it throws away information; in particular, the orientation of the edge. But knowing how the edges of our polygons are oriented will be vital in several proofs in this paper. Thus, we modify the notion of Hermite normal form so that it keeps the orientation information.
\begin{definition}
Let~$\sigma \subset \NQ$ be a (pointed, full-dimensional) cone. Denote the primitive ray generators of~$\sigma$ by~$\vec{v}_0$ and~$\vec{v}_1$, going anticlockwise. Let~$A_{\sigma}$ be the matrix with left column~$\vec{v}_0$ and right column~$\vec{v}_1$. Let~$H_{\sigma}$ be the (row-style) Hermite normal form of~$A_{\sigma}$. We call~$H_{\sigma}$ the~\emph{Hermite normal form} of the cone~$\sigma$. Equivalently, we can speak of the~\emph{Hermite normal form}~$H_E$ of an edge~$E$ of a Fano polygon by passing to the cone over~$E$.
\end{definition}
Note that the Hermite normal form of~$\sigma$ will be of the form~$H_{\sigma} =
\left(\begin{smallmatrix}
 1 & a \\
 0 & r
\end{smallmatrix}\right)$, for some~$r,a \in \ZZ$ with~$0 \le a < r$. We know that the first column of~$H_\sigma$ must be~$(1,0)$, since the ray generators of~$\sigma$ are primitive. For the same reason,~$a$ and~$r$ must be coprime.

We now describe the exact behaviour of our orientation-preserving Hermite normal form; in particular, we describe how it behaves when we apply unimodular transformations to the cone (or edge).
\begin{lemma}
Let~$\sigma \subset \NQ$ be a pointed, full-dimensional cone with Hermite normal form~$H_\sigma = \left(\begin{smallmatrix}
 1 & a \\
 0 & r
 \end{smallmatrix}\right)$. Let~$G$ be a unimodular matrix. If~$G$ has determinant~$1$, then~$G\sigma$ has Hermite normal form~$H_{G\sigma} = H_{\sigma}$. Instead suppose that~$G$ has determinant~$-1$. Then~$G\sigma$ has Hermite normal form~$H_{G\sigma} = \left(\begin{smallmatrix}
 1 & a^* \\
 0 & r
\end{smallmatrix}\right)$, where~$a a^* \congr 1 \mod r$.
\end{lemma}
\begin{proof}
First, suppose that~$G$ has determinant~$1$. So,~$G$ preserves the orientation of the ray generators of~$\sigma$. Thus,~$A_{G\sigma} = GA_{\sigma}$. Since the Hermite normal form of matrices is invariant under unimodular transformation, it follows that~$H_{G\sigma} = H_\sigma$.

Now, suppose that~$G$ has determinant~$-1$. Consider the cone~$\sigma' \coloneqq \left(\begin{smallmatrix}
 0 & 1 \\
 1 & 0
\end{smallmatrix}\right)G\sigma$. This is~$\SL_2(\ZZ)$-equivalent to~$\sigma$, and thus has the same Hermite normal form~$H_\sigma$. In particular,~$\sigma'$ can be transformed so that its primitive ray generators are the columns of~$H_\sigma$, which are~$(1,0)$ and~$(a, r)$. Applying the matrix~$\left(\begin{smallmatrix}
 0 & 1 \\
 1 & 0
\end{smallmatrix}\right)$ again, we see that~$G\sigma$ can be transformed so that it has primitive ray generators~$(r,a)$ and~$(0,1)$.
 
We now aim to compute the Hermite normal form of the matrix~$A \coloneqq \left(\begin{smallmatrix}
 r & 0 \\
 a & 1
\end{smallmatrix}\right)$. There is some unimodular matrix~$U$ such that~$A = U\left(\begin{smallmatrix}
 1 & a^* \\
 0 & r
\end{smallmatrix}\right)$ for some~$a^* \in \ZZ$ with~$0 \le a^* < r$. The first column of~$U$ must be~$(r,a)$. So, if the second column of~$U$ is~$(x,y)$, then we must satisfy~$r(a^* + x) = 0$ and~$a a^* + rw = 1$. In particular, in the second equation, notice that this gives~$a a^* \congr 1 \mod r$. Thus, since~$G\sigma$ is~$SL_2(\ZZ)$-equivalent to~$A$, the Hermite normal form of~$G\sigma$ is as described.
\end{proof}
We can recover the singularity type of~$\sigma$ from its associated Hermite normal form~$\left(\begin{smallmatrix}
 1 & a \\
 0 & r
\end{smallmatrix}\right)$. The singularity type is simply~$\frac{1}{r}(1,-a)$. Since cones are usually considered only up to~$\GL_2(\ZZ)$-equivalence, we can also get a Hermite normal form of~$\left(\begin{smallmatrix}
 1 & a^* \\
 0 & r
\end{smallmatrix}\right)$, where~$aa^* \congr 1 \mod r$, which results in a singularity type of~$\frac{1}{r}(1,-a^*)$. This is consistent, since these two singularities are in fact indistinguishable in the world of algebraic geometry.
\subsection{Mutations of lattice polygons}
Before we define the important notion of mutation, we introduce some notation. Let~$P \subset \NQ$ be Fano polytope and fix a primitive linear form~$w \in M$. Denote by~$N_h$ the~$h$-th graded piece of~$N$ with respect to~$w$, i.e.\ $N_h \coloneqq \setcond{\vec{x} \in N}{w(\vec{x}) = h}$. Denote by~$P_h$ the~$h$-th graded piece of~$P$ with respect to~$w$, i.e.\ $P_h \coloneqq \conv(P \cap N_h)$. We denote by~$\verts(P)$ the set of vertices of~$P$, and by~$\verts(P)_h$ the set of vertices of~$P$ at height~$h$ with respect to~$w$, i.e.\ $\verts(P)_h \coloneqq \verts(P) \cap N_h$. We write~$\hmin \coloneqq \min_{\vec{x} \in P} w(\vec{x})$ and~$\hmax \coloneqq \max_{\vec{x} \in P} w(\vec{x})$ so that all non-empty graded pieces of~$P$ with respect to~$w$ are contained in the list~$P_{\hmin}, P_{\hmin+1}, \ldots, P_{\hmax}$.
\begin{definition}[\!{\cite[Definition~3.3]{minkowski_mutation}}]
Let~$F \subset \NQ$ be a lattice polytope. Then~$F$ is said to be a~\emph{factor of~$P$ with respect to~$w$} if~$w(F) = \set{0}$ and, for all integers in the range~$\hmin \leq h < 0$, there exists a (possibly empty) lattice polytope~$R_h \subset \NQ$ such that~$\verts(P)_h \subseteq R_h + |h|F \subseteq P_h$. We refer to these lattice polytopes~$R_h$ as remainders.
\end{definition}
\begin{definition}[\!{\cite[Definition~3.3]{minkowski_mutation}}]
Let~$F \subset \NQ$ be a factor of~$P$ with respect to~$w$ with choice of remainders~$\Rc = \set{R_{\hmin}, \ldots, R_{-1}}$. Then the~\emph{mutation} of~$P$ with respect to~$w$ and~$F$ (and the choice of remainders~$\Rc$) is
\[
\mut_w\left(P,F; \Rc\right) \coloneqq \conv \left(\bigcup_{h=\hmin}^{-1} R_h \ \ \cup\ \ \bigcup_{h=0}^{\hmax} (P_h + hF) \right) \subset \NQ.
\]
\end{definition}
By~\cite[Proposition 3.8]{minkowski_mutation}, all choices of remainders~$\Rc$ will produce the same polytope; thus, we simply write~$\mut_w(P,F)$.
\begin{definition}
Let~$P$ and~$Q$ be Fano polytopes in~$\NQ$. We say that~$P$ and~$Q$ are~\emph{mutation-equivalent} if there exists a sequence of mutations taking~$P$ to~$Q$.
\end{definition}
Later on, when studying mutation-equivalence classes of Fano polygons, it will be useful to have some invariants at hand. The first invariant of interest is the Ehrhart series, whose definition we now recall.
\begin{definition}
Let~$P \subset \NQ$ be a polygon. The~\emph{Ehrhart series}~$\Ehr_P(t)$ of~$P$ is defined as the formal power series
\[
\Ehr_P(t) \coloneqq \sum_{k=0}^\infty |kP \cap N|\, t^k.
\]
\end{definition}
But instead of taking the Ehrhart series of a Fano polygon, we are interested in the Ehrhart series of its dual polygon. We recall the definition of (polar) duality of polygons.
\begin{definition}
Let~$P \subset \NQ$ be a polygon. Its~\emph{dual polyhedron} is defined to be the convex set
\[
P^* \coloneqq \setcond{u \in \MQ}{u(\vec{x}) \ge -1, \ \forall \vec{x} \in P}.
\]
\end{definition}
In~\cite{minkowski_mutation}, it was shown that if we apply a mutation to a Fano polytope~$P$, its dual polytope~$P^*$ is transformed by a piecewise linear function. Thus, it follows that the Ehrhart series of the dual polytope is invariant under mutation.
\begin{proposition}[{cf.\ \cite[Proposition~3.15]{minkowski_mutation}}]
Let~$P \subset \NQ$ be a Fano polygon. Then the Ehrhart series~$\Ehr_{P^*}(t)$ of its dual polygon~$P^*$ is invariant under mutation of~$P$.
\end{proposition}
In~\cite{ARTICLE:SING_CONT}, they introduced the important notion of singularity content of a Fano polygon. They proved that the singularity content of a Fano polygon is invariant under mutation. Here, we will introduce a modified version of this invariant. In particular, while the classical singularity content keeps track of singularity types of cones, we keep track of Hermite normal forms of cones. This modification will be vital to the success of later proofs.

We first define the notion locally, i.e.\ at the level of the edge (or cone).
\begin{definition} 
Let~$E$ be an edge of a Fano polygon with length~$\ell$ and height~$h$. Then~$\ell = nh + \widetilde{\ell}$, for some integers~$n \ge 0$ and~$0 \le \elltilde < h$. So, we can subdivide~$E$ into~$n$ segments of length~$h$ and a (possibly empty) segment~$\res(E)$ of length~$\elltilde$. We call~$\res(E)$ the~\emph{residue} of~$E$. Each segment of length and height~$h$ corresponds to a~\emph{primitive T-singularity}. If~$\elltilde > 0$, then the segment~$\res(E)$ corresponds to an~\emph{R-singularity}. We define the~\emph{directed singularity content}~$\SC(E)$ of~$E$ as the tuple~$(n_E, \res(H_E))$, where~$\res(H_E)$ is the~\emph{residue} of the Hermite normal form of~$E$, i.e.\ if~$\elltilde > 0$, then~$\res(H_E) = H_{\res(E)}$; else,~$\res(H_E) = \varnothing$.
\end{definition}
We now define the notion globally, i.e.\ at the level of the polygon.
\begin{definition} 
Let~$P \subset \NQ$ be a Fano polygon. The number~$n_P$ of primitive T-singularities of~$P$ is defined as the sum over all edges~$E$ of~$P$ of the number~$n_E$ of primitive T-singularities of~$E$. The~\emph{directed basket}~$\dirbasket_P$ of~$P$ is defined as the cyclically ordered set of non-empty residues of the Hermite normal forms of the edges of~$P$ (ordered anticlockwise). We define the~\emph{directed singularity content}~$\SC(P)$ of~$P$ as the tuple~$(n_P, \dirbasket_P)$. The subscript~$P$'s are often omitted when clear from context.
\end{definition}
Next, we prove that this modified version of singularity content is indeed a mutation-invariant of Fano polygons.
\begin{proposition}
Let~$P, Q \subset \NQ$ be Fano polygons. Suppose that they are related by a mutation, i.e.\ $Q = \mut_w(P,F)$ for some~$w \in M$ and factor~$F$ of~$P$ with respect to~$w$. Then, the directed singularity contents of~$P$ and~$Q$ coincide, i.e.\ $\SC(P) = \SC(Q)$.
\end{proposition}
\begin{proof}
Let~$\dirbasket_P \coloneqq \set{H_1, H_2, \ldots, H_m}$ be the directed basket of~$P$ and~$\dirbasket_Q \coloneqq \set{H_1', H_2', \ldots, H_{m'}'}$ be the directed basket of~$Q$. By~\cite[Proposition 3.6]{ARTICLE:SING_CONT}, the usual singularity content is invariant under mutation; therefore, we have~$m=m'$ and~$H_i' \in \set{H_i, H_i^*}$, for all~$i=1,2,\ldots,m$. It remains to show that~$H_i' = H_i$ for all~$i=1,2,\ldots,m$. But this follows from the fact that Hermite normal forms of edges are invariant under~$\SL_2(\ZZ)$-transformations; from~$P$ to~$Q$, the normals of the edges undergo one of two linear transformations which have determinant~$1$.
\end{proof}
In order to describe mutations, we introduce the notion of a long edge. This is already used implicitly in, for example,~\cite[Lemma~2.3]{minimality}. It is defined so that we can mutate a polygon with respect to some edge if and only if it is long.
\begin{definition} 
Let~$E$ be an edge with lattice length~$\ell$ and height~$h$. We say that~$E$ is~\emph{long} if its length is greater than or equal to its height, i.e.\ $\ell \ge h$. Otherwise, we say that~$E$ is~\emph{short}. In the latter case, the cone~$\sigma$ over~$E$ is its own residue, i.e.\ $\sigma = \res(\sigma)$. Finally, we say that an edge is~\emph{pure} if its residue is empty. Note that a pure edge must be long.
\end{definition}
We again recall the definition of minimality for Fano polygons, which we described in the introduction. 
\begin{definition}[\!{\cite[Definition~2.4]{minimality}}]
A Fano polygon~$P \subset \NQ$ is called~\emph{minimal} if, for every mutation~$Q \coloneqq \mut_w(P,F)$ of~$P$, we have~$|\boundary P \cap N| \le |\boundary Q \cap N|$, i.e.\ $P$ has the minimal number of boundary lattice points out of all polygons related to it by a single mutation.
\end{definition}
There are several equivalent conditions for a Fano polygon to be minimal (see~\cite[Lemma~4.2]{minimality}). We summarise the main characterisation that we use.
\begin{lemma}[\!{\cite[Corollary~4.5]{minimality}}]\label{lemma:min_poly_condition}
Let~$P \subset \NQ$ be a Fano polygon. Then~$P$ is minimal if and only if~$|\hmin| \le \hmax$, for each long edge~$E$ of~$P$, where~$u_E$ is the primitive normal to~$E$.
\end{lemma}
An important class of Fano polygons are reflexive polygons.
\begin{definition}[\!\cite{batyrev_reflexive}]
Let~$P \subset \NQ$ be a lattice polygon. We say that~$P$ is~\emph{reflexive} if each edge~$E$ of~$P$ has height~$1$.
\end{definition}
\begin{remark}
There are two properties of reflexive polygons relevant to minimality which we wish to highlight. First, it follows by Lemma~\ref{lemma:min_poly_condition} that all reflexive polygons are minimal (see~\cite[Example~4.4]{minimality}). Second, their edges are all long and pure. In particular, all reflexive polygons have empty basket.
\end{remark}
We finally introduce the notion of edge data, which is instrumental in the proof of Theorem~\ref{thm:sym_unique}.
\begin{definition}
Let~$P \subset \NQ$ be a Fano polygon. Order its edges anticlockwise:~$E_1, E_2, \ldots, E_m$. The~\emph{edge data}~$\Ec(P)$ of~$P$ is defined as the cyclically ordered set~$\set{H_1, H_2, \ldots, H_m}$, where~$H_i$ is the Hermite normal form of the edge~$E_i$.
\end{definition}
In general, the edge data is not a mutation-invariant. Note that if we take the residue of the edge data, and remove the empty entries, then we recover the directed basket, which~\emph{is} invariant under mutation.
\begin{example}\label{eg:edge_data}
Consider the Fano polygon
\[
P = \conv\set{\pm(5,1), \pm(5,6), \pm(4,7), \pm(-3,7), \pm(-4,5)},
\]
which is shown in Figure~\ref{fig:cs_fano_example}. This has edge data
\[
\Ec(P) = \set{\left(\begin{smallmatrix}
 1 & 6 \\
 0 & 25
\end{smallmatrix}\right), \left(\begin{smallmatrix}
 1 & 3 \\
 0 & 11
\end{smallmatrix}\right), \left(\begin{smallmatrix}
 1 & 36 \\
 0 & 49
\end{smallmatrix}\right), \left(\begin{smallmatrix}
 1 & 10 \\
 0 & 13
\end{smallmatrix}\right), \left(\begin{smallmatrix}
 1 & 23 \\
 0 & 29
\end{smallmatrix}\right), \left(\begin{smallmatrix}
 1 & 6 \\
 0 & 25
\end{smallmatrix}\right), \left(\begin{smallmatrix}
 1 & 3 \\
 0 & 11
\end{smallmatrix}\right), \left(\begin{smallmatrix}
 1 & 36 \\
 0 & 49
\end{smallmatrix}\right), \left(\begin{smallmatrix}
 1 & 10 \\
 0 & 13
\end{smallmatrix}\right), \left(\begin{smallmatrix}
 1 & 23 \\
 0 & 29
\end{smallmatrix}\right)}.
\]
\end{example}
Finally, we remark that edge data is invariant under~$\SL_2(\ZZ)$-transformation. Further, applying a transformation in~$\GL_2(\ZZ)$ with determinant~$-1$ to the polygon~$P$ changes the edge data in a predictable way. Write the edge data of~$P$ as~$\Ec(P) = \set{H_1, H_2, \ldots, H_m}$ and let~$U \in \GL_2(\ZZ) \setminus \SL_2(\ZZ)$. Then,~$\Ec(UP) = \set{H^*_m, \ldots, H^*_2, H^*_1}$. Of course, the analogous fact holds for the directed basket~$\dirbasket$ of~$P$.
\subsection{Symmetric and K{\"a}hler--Einstein polygons}
The main objects of study in this paper are symmetric and \KE\ Fano polygons. In this subsection, we state their definitions and several basic results concerning them.
\begin{definition}\label{def:KE}
Let~$P \subset \NQ$ be a Fano polygon. It's called~\emph{\KE} if the barycentre of its dual polytope~$P^*$ is the origin. It's called~\emph{symmetric} if its automorphism group~$\Aut(P)$ only fixes the origin, i.e.\ $\setcond{\vec{x} \in \NQ}{G \cdot \vec{x} = \vec{x},\, \forall G \in \Aut(P)} = \set{\vec{0}}$.
\end{definition}
\begin{definition}
Let~$P \subset \NQ$ be a polygon. We call it~\emph{centrally symmetric} if for all~$\vec{x} \in P$, we also have~$-\vec{x} \in P$. We call it~\emph{$3$-symmetric} if there exists an element~$G \in \Aut(P)$ of order~$3$.
\end{definition}
It is straightforward to see that if a polygon is centrally symmetric or~$3$-symmetric, then it is also symmetric. We show that the converse also holds.
\begin{proposition}
Let~$P \subset \NQ$ be a symmetric polygon. Then~$P$ is either centrally symmetric or~$3$-symmetric.
\end{proposition}
\begin{proof}
Consider the automorphism group of~$P$, which is a finite subgroup of~$\GL_2(\ZZ)$. Since~$P$ is symmetric,~$\Aut(P)$ must contain a rotation~$G$. By~\cite[Theorem~3]{mackiw1996finite},~$\Aut(P)$ is isomorphic to a subgroup of~$D_4$ or~$D_6$, the dihedral groups of order~$8$ and~$12$, respectively. Thus, the rotation~$G$ has order in~$\set{2,3,4,6}$. If~$G$ has order~$3$, then~$P$ is~$3$-symmetric. Otherwise,~$G$ has even order~$2g$. Since the element~$G^g$, which belongs to~$\Aut(P)$, is~$\left(\begin{smallmatrix}
 -1 & 0 \\
 0 & -1
\end{smallmatrix} \right)$, it follows that~$P$ is centrally symmetric.
\end{proof}
\begin{remark}\label{remark:times_notation}
The polygon of Example~\ref{eg:edge_data} is centrally symmetric. For cyclically ordered sets like directed baskets and edge data, let us introduce the notation~$\times g$ to indicate that the objects are repeated~$g$ times. For example,~$\set{H_1, H_2, H_3} \times 2$ means~$\set{H_1, H_2, H_3, H_1, H_2, H_3}$. It is a fact that the directed basket (and edge data) of a centrally symmetric polygon can always be written as~$\Cc \times 2$, for some cyclically ordered set~$\Cc$. The analogous fact holds for~$3$-symmetric polygons (replace~$\times 2$ with~$\times 3$).
\end{remark}
Next, let us recall the definition of weights and weight matrices of polygons. As we will immediately see, K{\"a}hler-Einstien Fano triangles have a nice describe in terms of weights. Further, we will give a description in~\S\ref{section:countereg} of a \KE\ quadrilateral~$P$ in terms of the weight system of~$P^*$.
\begin{definition}
Let~$P \subset \NQ$ be a polygon and fix an ordering of its vertices~$\vec{v}_1, \vec{v}_2, \ldots, \vec{v}_m$. Let~$\vec{q} \in \ZZ^m$ be an integer vector of length~$m$. Then~$\vec{q}$ is a~\emph{weight} of~$P$ if~$\sum_{i=1}^m q_i \vec{v}_i = \vec{0}$. Let~$W \in \ZZ^{(m-2) \times m}$ be an integer matrix of rank~$m-2$. Then~$W$ is a~\emph{weight matrix} for~$P$ if~$\sum_{i=1}^m W_{ji} \vec{v}_i = \vec{0}$ for all~$j=1,2,\ldots,m-2$.
\end{definition}
\begin{remark}
The toric variety corresponding to a Fano polygon~$P \subset \NQ$ is a fake weighted projective plane~$X_P = \PP(\lambda_0 \lambda_1, \lambda_2) / G$, where~$(\lambda_0, \lambda_1, \lambda_2)$ are the weights of~$P$ and~$G$ is the group~$N / N_P$ and~$N_P$ is the sublattice of~$N$ generated by the vertices of~$P$. Since we are in dimension two,~$G$ is a cyclic group of order~$k$.

There are typically several different fake weighted projective planes with the same weights~$(\lambda_0, \lambda_1, \lambda_2)$ and index~$k$. In some cases, though, there is only one such variety up to isomorphism. In these cases, we are safe to write~$X_P = \PP(\lambda_0, \lambda_1, \lambda_2) / \ZZ_k$. For instance, throughout the paper we will write~$X_P = \PP^2$ when~$P$ is isomorphic to the triangle with vertices~$(-1,-1)$,~$(1,0)$, and~$(0,1)$. We will also write~$X_P = \PP^2 / \ZZ_3$ when~$P$ is isomorphic to the triangle with vertices~$(-1,-1)$,~$(2,-1)$, and~$(-1,2)$.
\end{remark}
We prove the following statement, which will be useful in both~\S\ref{subsec:constrain_3sym} and~\S\ref{section:KE-triangles}.
\begin{lemma}\label{lemma:KE_triangle_weights}
Let~$P \subset \NQ$ be a \KE\ Fano triangle. Then~$P$ has weights~$(1,1,1)$.
\end{lemma}
\begin{proof}
It is well-known that the barycentre of a triangle is the average of its vertices. Therefore, a triangle has barycentre zero if and only if the sum of its vertices is the origin. Equivalently, a triangle has barycentre zero if and only if it has weights~$(1,1,1)$. By~\cite[Lemma~3.5]{conrads}, the weights of a Fano triangle coincide with the weights of its dual. Thus, a Fano triangle is \KE\ if and only if it has weights~$(1,1,1)$.
\end{proof}
We can also use the above to give an alternative proof for the following statement, which was proven in~\cite{on_KEs} using direct computation of the barycentre of the dual.
\begin{lemma}[\!{\cite[Proposition~4.1]{on_KEs}}]\label{lemma:KE_triangle_verts}
Let~$P \subset \NQ$ be a Fano triangle. Then~$P$ is \KE\ if and only if~$P$ is isomorphic to a triangle with vertices~$(-k,a-1)$,~$(k,-a)$, and~$(0,1)$, for some integers~$k, a \ge 1$ satisfying~$\gcd(k,a) = \gcd(k,a-1) = 1$.
\end{lemma}
\begin{proof}
Since~$P$ is Fano, we may assume without loss of generality that one of its vertices~$(0,1)$. We label its other vertices as~$(-m, b)$ and~$(k, -a)$, for some integers~$a,b,k,m \ge 1$. By Lemma~\ref{lemma:KE_triangle_weights}, since~$P$ is \KE, the sum of its vertices must be the origin. So,~$(k-m, 1+b-a) = (0,0)$. Thus,~$m = k$,~$b = a-1$, and the result follows.
\end{proof}
Finally, we must give the following lemma, which again is used in both~\S\ref{subsec:constrain_3sym} and~\S\ref{section:KE-triangles}.
\begin{lemma}\label{lemma:same_height_KE}
Let~$P \subset \NQ$ be a \KE\ Fano triangle. Suppose that all the edges of~$P$ have the same height~$h$ and are long. Then, either~$X_P = \PP^2$ or~$X_P = \PP^2 / \ZZ_3$.
\end{lemma}
\begin{proof}
By Lemma~\ref{lemma:KE_triangle_verts}, we may apply an appropriate transformation so that~$P$ has vertices~$(-k,a-1)$,~$(k,-a)$, and~$(0,1)$, for some integers~$a,k \ge 1$. For each edge~$E$ of~$P$, consider the triangle with base~$E$ and peak~$\vec{0}$. It has (normalised) volume~$k$. It also has volume~$\ell_E h_E$, where~$\ell_E$ is the length of~$E$ and~$h_E$ is the height of~$E$. But by assumption, all the edges of~$P$ have the same height~$h$. Thus, all edges have the same length~$\ell$ and so we obtain the following condition:
\begin{equation}\label{eq:lengths_KE}
\ell = \gcd(k,a+1) = \gcd(k,a-2) = \gcd(2k,2a-1).
\end{equation}
We see that~$\ell$ divides~$a+1$ and~$a-2$; thus,~$\ell$ divides~$3$.

If~$\ell = 1$ then, since each edge of~$P$ is long, we have~$h=1$. We may write~$a = 1$. It follows that~$X_P = \PP^2$. Otherwise,~$\ell = 3$. We may write~$a = 3b-1$, for some integer~$b \ge 1$. Plugging this into~\eqref{eq:lengths_KE}, we obtain:
\begin{equation}\label{eq:heights_KE}
1 = \gcd(h,b) = \gcd(h,b-1) = \gcd(2h,2b-1).
\end{equation}
Since the edges of~$P$ are long, we have~$h \le 3$. If~$h=1$, then we may write~$b=1$. It follows that~$X_P = \PP^2 / \ZZ_3$. If~$h=2$, then~\eqref{eq:heights_KE} implies that~$b \not\congr 0, 1 \mod 2$, which cannot hold. Else~$h=3$, and~\eqref{eq:heights_KE} implies that~$b \not\congr 0,1,2 \mod 3$, which also cannot hold.
\end{proof}
\section{At most one symmetric}\label{section:most_one_cs_fano}
We dedicate this section to proving part of Theorem~\ref{thm:main}. Namely, we prove the following statement.
\begin{theorem}\label{thm:sym_unique}
There is at most one symmetric Fano polygon in each mutation-equivalence class.
\end{theorem}
As we will see later in Proposition~\ref{prop:empty_basket}, the above theorem holds for symmetric Fano polygons with empty basket~$\Bc$. Thus, the next few subsections will mainly focus on the case~$\Bc \neq \varnothing$.
\subsection{Observation of both behaviours}\label{subsec:both_behaviours}
Before we prove Theorem~\ref{thm:sym_unique}, we will first demonstrate that both possibilities occur. More precisely, we show that there exists a mutation-equivalence class with no symmetric Fano polygons or \KE\ Fano triangles and that there exists a mutation-equivalence class with exactly one of them.
\begin{example}\label{eg:no_KE}
Consider the Fano polygon~$P \coloneqq \conv\set{(2,-3), (1,5), (-1,-2)} \subset \NQ$. Since~$P$ is a triangle with weights~$(3,7,13)$, it is not \KE\ by Lemma~\ref{lemma:KE_triangle_weights}. Further, the edges of~$P$ all have length~$1$ while their heights are~$3$,~$7$, and~$13$ --- all greater than~$1$. Thus, none of the edges of~$P$ are long. Therefore,~$P$ is the only polygon in its mutation-equivalence class. So, the mutation-equivalence class of~$P$ contains no \KE\ polygons. In particular, it contains no symmetric Fano polygons or \KE\ Fano triangles.
\end{example}
We now show that the other possibility can occur.
\begin{example}\label{eg:one_KE}
Consider the Fano polygon~$P \coloneqq \conv\set{(\pm(1,-2), \pm(2,-1), \pm(1,1)} \subset \NQ$. Clearly,~$P$ is centrally symmetric (and thus symmetric). The edges of~$P$ all have length~$1$ and height~$3$; thus, none of the edges of~$P$ are long. Therefore,~$P$ is the only polygon in its mutation-equivalence class. So, the mutation-equivalence class of~$P$ contains exactly one \KE\ polygon. In particular, it contains exactly one symmetric Fano polygon or \KE\ Fano polygon.
\end{example}
In the following subsections, we will show that for symmetric Fano polygons, no other possibility can occur, i.e.\ that if two symmetric Fano polygons are mutation-equivalent, then they are isomorphic.
\subsection{Constraints on centrally symmetric Fano polygons}\label{subsec:constrain_cs}
When mutating polygons, it's the long edges (in particular, the T-singularities) which move and change, while the short edges (i.e.\ the R-singularities) stay the same. Generally speaking, if a polygon has more long edges, then its behaviour under mutation will be richer. As we saw in Example~\ref{eg:no_KE} and Example~\ref{eg:one_KE}, if a polygon has no long edges, then its mutation-equivalence class is trivial; we can conclude that Theorem~\ref{thm:main} holds in this case. For centrally symmetric polygons, the next simplest case is when the polygon has exactly one pair of parallel long edges.
\begin{proposition}\label{prop:one_long_edge_cs}
Let~$P \subset \NQ$ be a centrally symmetric polygon with exactly one pair of long edges~$\pm E$. Let~$Q \subset \NQ$ be a symmetric polygon. Suppose~$P$ is mutation-equivalent to~$Q$. Then~$P \isom Q$.
\end{proposition}
\begin{proof}
Let~$w \in M$ be the primitive inner normal to~$E$. Then~$-w$ is the primitive inner normal to~$-E$. Consider~$\mut_w(P,F)$, where~$F$ is a factor of~$P$ with respect to~$w$. Then again, this polygon will have at most two long edges, whose primitive inner normals are in~$\set{w, -w}$. Thus, by induction, all polygons mutation-equivalent to~$P$ will have at most two long edges with inner normals in~$\set{w, -w}$.

Now consider the symmetric polygon~$Q$, which is mutation-equivalent to~$P$. Since the number of primitive T-singularities is invariant under mutation, and~$P$ has at least one primitive T-singularity, it follows that~$Q$ has at least one long edge. If~$Q$ were~$3$-symmetric, then it would have at least three long edges. So, since~$Q$ has at most two long edges, it must be centrally symmetric. In particular, the number of primitive T-singularities on its two long edges must be equal. Thus,~$Q$ is isomorphic to~$P$.
\end{proof}
Now, we are left to consider centrally symmetric Fano polygons with at least two pairs of parallel long edges. The rest of the subsection is dedicated to constraining the combinatorics of these polygons. In fact, in Corollary~\ref{cor:num_edges_cs}, we discover that these polygons must have \emph{exactly} two pairs of parallel long edges with at most~$8$ primitive T-singularities in total. This restricts the complexity of their behaviour under mutation, making the study of these objects under mutation realistic.
\begin{lemma}\label{lemma:det_one}
Let~$P \subset \NQ$ be a centrally symmetric Fano polygon with non-empty basket~$\Bc$ and two pairs of opposite long edges~$\pm E_1$ and~$\pm E_2$. Then the primitive normals to~$E_1$ and~$E_2$ span the dual lattice~$M$.
\end{lemma}
\begin{proof}
We may apply an appropriate unimodular transformation so that the primitive inner normal to~$E_1$ is~$u_1 = \ed{1}$. We may also assume that the edge~$E_2$ is on the bottom. Thus, its primitive inner normal vector~$u_2 \coloneqq (a,b)^t \in M$ will have~$b > 0$, as in Figure~\ref{fig:2_long_edges}. We obtain~$\det(u_1, u_2) = b$. Thus,~$u_1$ and~$u_2$ span~$M$ if and only if~$b=1$. So, assume towards a contradiction that~$b \ge 2$.
 
Consider the primitive direction vector for~$E_2$, which is~$(b,-a)$. We can see that the starting vertex of~$E_2$ uses up one lattice point, and each unit segment of~$E_2$ uses up~$b$ lattice points. Since there are~$2h_1 + 1$ lattice points between~$E_1$ and~$-E_1$ for~$E_2$ to go through, we obtain the inequality~$b \ell_2 \le 2h_1$. We can similarly derive that~$b \ell_1 \le 2h_2$. Now, since~$E_2$ is a long edge and~$E_1$ has a smaller height than~$E_2$, we obtain that~$b h_1 \le 2 h_1$. Thus,~$b = 2$ and the inequalities become equalities, giving~$\ell_1 = h_1 = h_2 = \ell_2$. Further, all~$2h_1 + 1$ lattice points were used up by~$E_2$. This means that no more lattice points are left for any other edges going from~$E_1$ to~$-E_1$ to go through. So, no more pairs of edges exist other than~$\pm E_1$ and~$\pm E_2$. Thus,~$P$ is a centrally symmetric quadrilateral with only T-singularities, which contradicts our assumption that the basket~$\Bc$ is non-empty.
\end{proof}
\begin{figure}[h]
\centering
\begin{subfigure}{.45\textwidth}
 \begin{tikzpicture}[scale=0.5][line cap=round,line join=round,>=triangle 45,x=1cm,y=1cm]
 \begin{axis}[
 x=1cm,y=1cm,
 axis lines=middle,
 ymajorgrids=true,
 xmajorgrids=true,
 xmin=-6,
 xmax=6,
 ymin=-10,
 ymax=4,
 xtick={-6,-9,...,6},
 ytick={-10,-9,...,4}
 ]
 
 \draw [line width=2pt] (-5,3)-- (-5,-2);
 \draw [line width=2pt] (-4,-3)-- (2,-9);
 \draw [line width=2pt] (5,2)-- (5,-3);
 
 \begin{scriptsize}
 \draw [fill=black] (-5,3) circle (2.5pt);
 \draw [fill=black] (-5,-2) circle (2.5pt);
 \draw[color=black] (-5.5, 1.5) node {\huge{$E_1$}};
 \draw [fill=black] (-4,-3) circle (2.5pt);
 \draw [fill=black] (2,-9) circle (2.5pt);
 \draw[color=black] (-1.4, -6.3) node {\huge{$E_2$}};
 \draw [fill=black] (5,2) circle (2.5pt);
 \draw [fill=black] (5,-3) circle (2.5pt);
 \draw[color=black] (5.5, -1.5) node {\huge{$-E_1$}};
 \end{scriptsize}
 \end{axis}
 \end{tikzpicture}
 \caption{A configuration of two long edges of a centrally symmetric polygon; see the proof of Lemma~\ref{lemma:det_one}}
 \label{fig:2_long_edges}
\end{subfigure}
\begin{subfigure}{.45\textwidth}
 \begin{tikzpicture}[scale=0.5][line cap=round,line join=round,>=triangle 45,x=1cm,y=1cm]
 \begin{axis}[
 x=1cm,y=1cm,
 axis lines=middle,
 ymajorgrids=true,
 xmajorgrids=true,
 xmin=-8,
 xmax=8,
 ymin=-6,
 ymax=6,
 xtick={-8,...,8},
 ytick={-6,...,6}]
 
 \fill[line width = 2pt, color = blue, fill = blue, fill opacity = 0.05] (1,5) -- (6,5) -- (7,4) -- (7,-3) -- (5,-4) -- (-1,-5) -- (-6,-5) -- (-7,-4) -- (-7,3) -- (-5,4) -- cycle;
 \draw [line width=2pt,color=black] (1,5) -- (6,5);
 \draw [line width=2pt,color=black] (6,5) -- (7,4);
 \draw [line width=2pt,color=black] (7,4) -- (7,-3);
 \draw [line width=2pt,color=black] (7,-3) -- (5,-4);
 \draw [line width=2pt,color=black] (5,-4) -- (-1,-5);
 \draw [line width=2pt,color=black] (-1,-5)-- (-6,-5);
 \draw [line width=2pt,color=black] (-6,-5)-- (-7,-4);
 \draw [line width=2pt,color=black] (-7,-4)-- (-7,3);
 \draw [line width=2pt,color=black] (-7,3)-- (-5,4);
 \draw [line width=2pt,color=black] (-5,4)-- (1,5);
 
 \begin{scriptsize}
 \draw [fill=black] (1,5) circle (2.5pt);
 \draw [fill=black] (6,5) circle (2.5pt);
 \draw [fill=black] (7,4) circle (2.5pt);
 \draw [fill=black] (7,-3) circle (2.5pt);
 \draw [fill=black] (5,-4) circle (2.5pt);
 \draw [fill=black] (-1,-5) circle (2.5pt);
 \draw [fill=black] (-6, -5) circle (2.5pt);
 \draw [fill=black] (-7, -4) circle (2.5pt);
 \draw [fill=black] (-7, 3) circle (2.5pt);
 \draw [fill=black] (-5, 4) circle (2.5pt);
 \draw[color=black] (3.5, 5.5) node {\huge{$-E_2$}};
 \draw[color=black] (7.5, 1.5) node {\huge{$-E_1$}};
 \draw[color=black] (-3.5, -5.5) node {\huge{$E_2$}};
 \draw[color=black] (-7.5, -1.5) node {\huge{$E_1$}};
 \end{scriptsize}
 \end{axis}
 \end{tikzpicture}
 \caption{An example of a centrally symmetric Fano polygon with exactly two pairs of long edges~$\pm E_1, \pm E_2$ whose primitive normals span the dual lattice~$M$. Its edge data is described in Example~\ref{eg:edge_data}.}
 \label{fig:cs_fano_example}
\end{subfigure}
\end{figure}
\begin{remark}
\begin{enumerate}
\item
Let~$P \subset \NQ$ be a centrally symmetric Fano polygon with two pairs of long edges~$\pm E_1$ and~$\pm E_2$ whose primitive inner normal vectors are~$\pm u_1$ and~$\pm u_2$, respectively. If~$\Bc_P \neq \varnothing$ then, by Lemma~\ref{lemma:det_one},~$\det(u_1, u_2) = 1$. The matrix~$V$ with rows~$u_1$ and~$u_2$ is invertible in~$\GL_2(\ZZ)$. Thus, we can apply a unimodular transformation to~$P$ so that the primitive inner normals to~$E_1$ and~$E_2$ become~$\ed{1}$ and~$\ed{2}$, respectively.
\item
Later, in~\S\ref{subsec:proof_sym}, we will be working with the edge data~$\Ec(P)$ and directed basket~$\dirbasket_P$ of~$P$. Thus, we will want to only transform~$P$ by elements of~$\SL_2(\ZZ)$ so that this data stays the same. If the matrix~$V$ has determinant~$-1$, we instead simply consider the matrix~$V'$ with rows~$u_1$ and~$-u_2$. This now has determinant~$1$, and so we can apply an~$\SL_2(\ZZ)$-transformation to~$P$ so that the primitive inner normals to~$E_1$ and~$-E_2$ become~$\ed{1}$ and~$\ed{2}$, respectively.
\end{enumerate}
\end{remark}
We now give the following lemma, which helps us to prove Proposition~\ref{prop:heights_fixed_cs}.
\begin{lemma}\label{lemma:height_bound}
Let~$P \subset \NQ$ be a centrally symmetric Fano polygon with non-empty basket~$\Bc \neq \varnothing$. Suppose that~$P$ has two pairs of opposite long edges~$\pm E_1$ and~$\pm E_2$ with heights~$h_1 \le h_2$. Then,~$2 \le h_1 \le h_2 < 2h_1$.
\end{lemma}
\begin{proof}
We may assume without loss of generality that the primitive inner normals to~$E_1$ and~$E_2$ are~$\ed{1}$ and~$\ed{2}$, respectively. We first show that~$h_1 \ge 2$. Assume towards a contradiction that~$h_1 = 1$. Then~$P$ is contained in the strip~$\set{-1 \le x \le 1}$. So,~$E_2$ has length at most~$2$. Now, since~$E_2$ is long, it must have height~$h_2 = 1$ or~$h_2 = 2$. We explore both possibilities.

Suppose that~$h_2 = 1$. Then~$P$ is contained in the box~$\set{-1 \le x, y \le 1}$. Up to isomorphism, there are two polygons in this box which have long edges~$E_1$ and~$E_2$ with the prescribed heights: either~$P = \conv\set{\pm(1,1), \pm(-1,1)}$ or~$P = \conv\set{\pm(1,0), \pm(1,1), \pm(0,1)}$. In either case,~$P$ has an empty basket~$\Bc = \varnothing$, which is a contradiction.

Now suppose that~$h_2 = 2$. Then~$P$ must be the quadrilateral~$\conv\set{\pm(1,2), \pm(-1,2)}$. This has empty basket~$\Bc = \varnothing$, so we again reach a contradiction. Thus, we have shown that~$h_1 \ge 2$.

Finally, we show that~$h_2 < 2h_1$. Note that~$P$ is contained in the strip~$\set{-h_1 \le x \le h_1}$, which implies that~$E_2$ has length~$\ell_2 \le 2h_1$. Since~$E_2$ is a long edge, we obtain that~$h_2 \le 2h_1$. We see that it now remains to show that~$h_2 \neq 2h_1$.

Suppose towards a contradiction that~$h_2 = 2h_1$. Then~$P$ must be the rectangle with vertices~$\pm(h_1, 2h_1)$ and~$\pm(-h_1, 2h_1)$. Since~$P$ is Fano, its vertices must be primitive, i.e.\ $\gcd(h_1, 2h_1) = 1$. Thus,~$h_1 = 1$. But we have already seen above that this results in a contradiction. Therefore,~$h_2 < 2h_1$.
\end{proof}
In order to decide whether two polygons~$P$ and~$Q$ are mutation-equivalent, we can compare their mutation-invariants. One such invariant is the Ehrhart series of the dual polygon, i.e.\ $\Ehr_{P^*}(t) = \Ehr_{Q^*}(t)$. We find a striking pattern in the coefficients of the Ehrhart series.
\begin{proposition}\label{prop:heights_fixed_cs}
Let~$P \subset \NQ$ be a centrally symmetric Fano polygon with non-empty basket of R-singularities. Suppose~$P$ has two non-parallel long edges~$E_1$ and~$E_2$, which have primitive inner normals~$u_1$ and~$u_2$ and heights~$h_1$ and~$h_2$, respectively, with~$h_1 \le h_2$. Then,
\[
kP^* \cap M =
\begin{cases}
 \set{0}, & 0 \le k < h_1 \\
 \set{0, \pm u_1}, & h_1 \le k < h_2 \\
 \set{0, \pm u_1, \pm u_2}, & k = h_2.
\end{cases}
\]
\end{proposition}
\begin{proof}
The inclusion ``$\supseteq$'' is clear; thus, it is enough to show that~$h_2 P^* \cap M \subseteq \set{0, \pm u_1, \pm u_2}$. By Lemma~\ref{lemma:det_one}, we may assume that~$u_1 = (1,0)^t$ and~$u_2 = (0,1)^t$. We first show that~$\pm(1,1)^t \not\in h_2 P^*$. This is equivalent to showing that there exists a point~$(x,y) \in P$ satisfying~$x+y < -h_2$. Let~$(a,-h_2)$ be the left vertex of~$E_2$, where~$a$ is an integer satisfying~$-h_1 \le a \le h_1 - h_2 \le 0$. In fact,~$a < 0$. Otherwise,~$(0,-h_2)$ would be a vertex of~$P$. By Lemma~\ref{lemma:height_bound},~$h_2 \ge 2$. On the other hand, since~$P$ is Fano, its vertices must be primitive. So, we must have that~$h_2 = 1$, which is a contradiction. Thus, we have~$(a,-h_2) \in P$ satisfying~$a < 0$. Equivalently,~$a - h_2 < -h_2$. We can conclude that~$(1,1)^t$ is not in~$h_2 P^*$.

Using the same logic as above, we can also show that~$\pm(-1,1)^t \not\in h_2 P^*$. Furthermore, since~$(1,1)^t \not\in h_2P^*$, we can rule out from~$h_2 P^*$ all points in the region~$(1,1)^t + \cone((1,0)^t, (0,1)^t)$. We can rule out similarly defined regions from the other 3 quadrants; see Figure~\ref{fig:ehr_cs_proof}. It now remains to treat the points on the axes.

\begin{figure}[ht]
    \centering
    \begin{tikzpicture}[scale=0.65][line cap=round,line join=round,x=1cm,y=1cm]
        \begin{axis}[
            x=1cm,y=1cm,
            axis lines=middle,
            ymajorgrids=true,
            xmajorgrids=true,
            xmin=-3,
            xmax=3,
            ymin=-3,
            ymax=3,
            xtick={-2,...,2},
            ytick={-2,...,2}
            ]
            \clip(-3,-3) rectangle (3,3);
            
            \fill[fill=red,fill opacity=0.2] (0,1) -- (1,0) -- (0,-1) -- (-1,0) -- cycle;
            
            \fill[fill=gray,fill opacity=0.1] (-1,-1) -- (-5,-1) -- (-5,-5) -- (-1,-5) -- cycle;
            \fill[fill=gray,fill opacity=0.1] (-1,1) -- (-5,1) -- (-5,5) -- (-1,5) -- cycle;
            \fill[fill=gray,fill opacity=0.1] (1,1) -- (5,1) -- (5,5) -- (1,5) -- cycle;
            \fill[fill=gray,fill opacity=0.1] (1,-1) -- (1,-5) -- (5,-5) -- (5,-1) -- cycle;
            
            \draw [line width=2pt,color=red] (0,1)-- (1,0);
            \draw [line width=2pt,color=red] (1,0)-- (0,-1);
            \draw [line width=2pt,color=red] (0,-1)-- (-1,0);
            \draw [line width=2pt,color=red] (-1,0)-- (0,1);
            
            \draw [line width=2pt] (-1,-1)-- (-5,-1);
            \draw [line width=2pt] (-5,-5)-- (-1,-5);
            \draw [line width=2pt] (-1,-5)-- (-1,-1);
            \draw [line width=2pt] (-1,1)-- (-5,1);
            \draw [line width=2pt] (-1,5)-- (-1,1);
            \draw [line width=2pt] (1,1)-- (5,1);
            \draw [line width=2pt] (1,5)-- (1,1);
            \draw [line width=2pt] (1,-1)-- (1,-5);
            \draw [line width=2pt] (5,-1)-- (1,-1);
            
            \begin{scriptsize}
                \draw [fill=black] (1,1) circle (2.5pt);
                \draw [fill=black] (-1,1) circle (2.5pt);
                \draw [fill=black] (-1,-1) circle (2.5pt);
                \draw [fill=black] (1,-1) circle (2.5pt);
            \end{scriptsize}
        \end{axis}
    \end{tikzpicture}
    \caption{In red: the polygon~$\conv\set{\pm u_1, \pm u_2}$.
    In grey: the regions which can be ruled out from~$h_2 P^*$ given that their apexes are not in~$h_2 P^*$.}
    \label{fig:ehr_cs_proof}
\end{figure}

Clearly, the only points of~$h_2 P^*$ on the~$y$-axis are~$0$ and~$\pm u_2$. By Lemma~\ref{lemma:height_bound}, we have~$h_2 < 2h_1$. Therefore, the only points of~$h_2 P^*$ on the~$x$-axis are~$0$ and~$\pm u_1$. Thus, we now have the desired result.
\end{proof}
An immediate implication of Proposition~\ref{prop:heights_fixed_cs} is that the heights of the long edges can be read off the Ehrhart series.
\begin{example}
Consider the polygon~$P \subset \NQ$ from Example~\ref{eg:edge_data}. Let~$Q \subset \NQ$ be a centrally symmetric Fano polygon mutation-equivalent to~$P$. Then it must also have at least two pairs of opposite long edges whose primitive normals span the lattice~$M$. The dual polygon~$P^*$ has Ehrhart series~$\Ehr_{P^*}(t) = 1 + t + t^2 + t^3 + t^4 + 3t^5 + 3t^6 + 5t^7 + O(t^8)$. Since~$Q$ is mutation-equivalent to~$P$, its dual~$Q^*$ has the same Ehrhart series. Since the coefficient jumps from~$1$ to~$3$ at the~$t^5$ term and then to~$5$ at the~$t^7$ term, it follows that~$Q$ has long edges with heights~$5$ and~$7$, just as~$P$ does. This is because, by Proposition~\ref{prop:heights_fixed_cs}, the non-origin lattice points in~$7Q^* \cap M$ must be primitive normals to the long edges of~$Q$.
\end{example}
The main implications of Proposition~\ref{prop:heights_fixed_cs} are summarised in the following corollary.
\begin{corollary}\label{cor:num_edges_cs}
Let~$P$ be a centrally symmetric Fano polygon with non-empty basket. Suppose that~$P$ has at least two pairs of long edges~$\pm E_1$ and~$\pm E_2$. Then the long edges of~$P$ are exactly~$\pm E_1$ and~$\pm E_2$. Moreover, if~$P$ has an R-singularity~$H_{\sigma} \in \dirbasket_P$ with the same height as one of the long edges, then~$\sigma$ will be on a long edge of~$P$, i.e.\ $H_{\sigma} = \res(H_{E_i})$ for some~$i=1,2$. In particular, the number of edges is fixed.
\end{corollary}
\begin{proof}
Without loss of generality, suppose that the height~$h_1$ of~$E_1$ is the smallest out of all long edges of~$P$ and the height~$h_2$ of~$E_2$ is the largest out of all long edges of~$P$. By Proposition~\ref{prop:heights_fixed_cs}, there are no edges of~$P$ other than~$\pm E_1$ and~$\pm E_2$ which have height between~$1$ and~$h_2$, inclusively; otherwise, the primitive inner normal of such an edge would be a lattice point of~$h_2 P^* \cap M$. Thus, we may conclude that~$\pm E_1$ and~$\pm E_2$ are the only long edges of~$P$ and that all other edges of~$P$ must have height strictly greater than~$h_2$. It now follows that if an R-singularity of~$P$ has height~$h_1$ or~$h_2$, then it must be on one of the long edges of~$P$.
\end{proof}
\subsection{Constraints on 3-symmetric Fano polygons}\label{subsec:constrain_3sym}
As stated at the beginning of~\S\ref{subsec:constrain_cs}, the more long edges a polygon has, the richer its behaviour under mutation will be. The simplest non-trivial case for~$3$-symmetric polygons is when there is~\emph{at least one} triple of long edges. In fact, similarly to the centrally symmetric case, we see that these polygons must have~\emph{exactly one} triple of long edges (Corollary~\ref{cor:num_edges_3sym}).

Since we will reuse the following result in the general setting, we emphasise that~$P$ can have either empty or non-empty basket of R-singularities in the following lemma.
\begin{lemma}\label{lemma:ord3_det}
Let~$P \subset \NQ$ be a~$3$-symmetric Fano polygon with~$G \in \Aut(P)$ having order~$3$. Suppose~$P$ has a long edge~$E$ with primitive inner normal~$u$. Then, one of the following must hold:
\begin{itemize}
\item[(i)] $\det(u,uG) = 3$ and~$X_P = \PP^2$;
\item[(ii)] $\det(u,uG) = 1$.
\end{itemize}
\end{lemma}
\begin{proof}
Denote by~$h_E$ the height of~$E$. Consider the (not necessarily lattice) triangle~$Q$ which is the intersection of the supporting half-spaces of~$E$,~$GE$, and~$G^2 E$, i.e.
\[
Q \coloneqq \setcond{\vec{x} \in \NQ}{u(\vec{x}) \ge -h_E, (uG)(\vec{x}) \ge -h_E, (uG^2)(\vec{x}) \ge -h_E}.
\]
Due to the action of~$G$ on~$Q$, there exists a unique~$t \in \QQ_{>0}$ such that~$tQ$ has primitive vertices. Since~$Q$ contains the origin, we can conclude that~$tQ$ is Fano. Now, consider the sum of the vertices of~$tQ$. This is fixed by~$G$, so it must equal the origin. Thus,~$tQ$ has weights~$(1,1,1)$.

Let~$F$ be an edge of~$tQ$. Then the other two edges of~$tQ$ are~$GF$ and~$G^2 F$. In particular, the edges of~$tQ$ all have the same height. Further, since the original edge~$E$ of~$P$ is long, the edges~$F$,~$GF$, and~$G^2 F$ must also be long. So, we may apply Lemma~\ref{lemma:same_height_KE} to~$tQ$. We obtain that either (i)~$X_{tQ} = \PP^2$ or (ii)~$X_{tQ} = \PP^2 / \ZZ_3$.

In case (i), the result follows from the observation that~$t=1$ and~$P = Q$. In case (ii), the result follows after noticing that~$E$ and~$F$ have the same primitive inner normal vector.
\end{proof}
The above lemma is enough to prove that~$3$-symmetric Fano polygons are minimal, which we do later in Lemma~\ref{lemma:P_minimal}. For the rest of this subsection, we again focus on~$3$-symmetric Fano polygons with~\emph{non-empty} baskets. The above lemma also constrains the dual polygon~$P^*$. As in~\S\ref{subsec:constrain_cs}, we next describe the mutation-invariants~$|kP^* \cap M|$ for relevant~$k$.
\begin{proposition}\label{prop:heights_fixed_3_sym}
Let~$P \subset \NQ$ be a~$3$-symmetric Fano polygon with element~$G \in \Aut(P)$ of order~$3$.
Suppose that~$P$ has a non-empty basket of singularities and that~$P$ has a long edge~$E$ with primitive inner normal~$u$ and height~$h$. Then
\[
kP^* \cap M = \begin{cases}
 \set{0}, & 0 \le k < h \\
 \set{0, u, uG, uG^2}, & k = h.
\end{cases}
\]
\end{proposition}
\begin{proof}
Since the basket of R-singularities is non-empty, we are in case (ii) of Lemma~\ref{lemma:ord3_det}. Thus, we may assume, after applying a suitable linear transformation to~$P$, that~$u = (0,1)^t$ and~$uG = (1,0)^t$. This determines~$G$, and so~$uG^2 = (-1,-1)^t$. Now, since the edges~$E$,~$GE$, and~$G^2 E$ all have height~$h$, we see that~$u, uG, uG^2 \not\in kP^*$, for~$0 \le k < h$, and~$\set{0,u,uG,uG^2} \subseteq hP^* \cap M$.

In order to show the reverse inclusion, it is enough to show that~$(1,1)^t \not\in hP^*$, as we will now explain. Consider the cone~$C$ in~$\MQ$ generated by~$(1,0)^t$ and~$(0,1)^t$. It is clear that, for~$m > 1$, the points~$(m,0)^t$ and~$(0,m)^t$ do not belong to~$hP^*$. Now let~$w \in (1,1)^t + C$. It is clear that~$(1,1)^t$ is contained in~$\conv\set{(1,0)^t, (0,1)^t, w}$. So, if~$w \in hP^*$, then~$(1,1)^t \in hP^*$. Therefore, if we can show that~$(1,1)^t \not\in hP^*$, then~$hP^* \cap C \cap M = \set{0,u,uG}$. We can treat~$CG$ and~$CG^2$ similarly, and reach the desired conclusion.

\begin{figure}[ht]
\centering
\begin{tikzpicture}[scale=0.7][line cap=round,line join=round,x=1cm,y=1cm]
 \begin{axis}[
 x=1cm,y=1cm,
 axis lines=middle,
 ymajorgrids=true,
 xmajorgrids=true,
 xmin=-3,
 xmax=3,
 ymin=-2,
 ymax=2,
 xtick={-2,...,2},
 ytick={-1,0,1}
 ]
 \fill[line width=2pt,color=red,fill=red,fill opacity=0.1] (0,1) -- (1,0) -- (-1,-1) -- cycle;
 
 \fill[line width=2pt,fill=gray,fill opacity=0.1] (-1,0) -- (-1,2) -- (-3,2) -- (-3,-2) -- cycle;
 \fill[line width=2pt,fill=gray,fill opacity=0.1] (1,1) -- (1,2) -- (3,2) -- (3,1) -- cycle;
 \fill[line width=2pt,fill=gray,fill opacity=0.1] (0,-1) -- (-1,-2) -- (3,-2) -- (3,-1) -- cycle;
 
 \draw [line width=2pt,color=red] (0,1)-- (1,0);
 \draw [line width=2pt,color=red] (1,0)-- (-1,-1);
 \draw [line width=2pt,color=red] (-1,-1)-- (0,1);
 
 \draw [line width=2pt] (-1,0)-- (-1,2);
 \draw [line width=2pt] (-3,-2)-- (-1,0);
 \draw [line width=2pt] (1,1)-- (1,2);
 \draw [line width=2pt] (3,1)-- (1,1);
 \draw [line width=2pt] (0,-1)-- (-1,-2);
 \draw [line width=2pt] (3,-1)-- (0,-1);
 \begin{scriptsize}
 \draw [fill=black] (1,1) circle (2.5pt);
 \draw [fill=black] (0,-1) circle (2.5pt);
 \draw [fill=black] (-1,0) circle (2.5pt);
 \end{scriptsize}
 \end{axis}
\end{tikzpicture}
\caption{In red: the polygon~$\conv\set{u,uG,uG^2}$.
In grey: the regions which can be ruled out from~$hP^*$ given that their apexes are not in~$hP^*$.}
\label{fig:ehr_3_proof}
\end{figure}
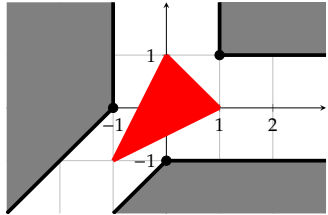

It remains to prove that~$(1,1)^t$ is not contained in~$hP^*$. This is equivalent to showing that there exists a point~$(x,y) \in P$ such that~$x+y < -h$. Let~$(a,-h)$ be a vertex of~$E$. Without loss of generality, and due to~$E$ being a long edge, we know that~$(a+h,-h) \in E$. So, we get the point~$G \cdot (a+h,-h) = (-h,-a)$ of~$P$. Then, we may set
\[
(x,y) \coloneqq \begin{cases}
 (a,-h), & a < 0 \\
 (-h,-a), & a > 0.
\end{cases}
\]
The case~$a=0$ cannot occur. Otherwise, since~$P$ is a Fano polytope, its vertex~$(a,-h)$ must be primitive. It follows that~$h=1$. Now,~$P$ has exactly one interior lattice point, i.e.\ it is reflexive. But this implies that~$P$ has an empty basket of R-singularities, which contradictions our starting assumption. Thus, we may conclude that indeed~$(1,1)^t$ does not belong to~$hP^*$.
\end{proof}
Let us now summarise the main implications of Proposition~\ref{prop:heights_fixed_3_sym}. Note that this is analogous to Corollary~\ref{cor:num_edges_cs}.
\begin{corollary}\label{cor:num_edges_3sym}
Let~$P$ be a~$3$-symmetric Fano polygon with non-empty basket. Suppose that~$P$ has at least one triple of long edges~$E$,~$GE$, and~$G^2 E$, where~$G \in \Aut(P)$ is an element of order~$3$. Then the long edges of~$P$ are exactly~$E$,~$GE$, and~$G^2 E$. Moreover, if~$P$ has an R-singularity~$H_{\sigma} \in \dirbasket_P$ with the same height as one of the long edges, then~$\sigma$ will be on a long edge of~$P$, i.e.\ $H_{\sigma} = \res(H_E)$. In particular, the number of edges is fixed.
\end{corollary}
\begin{proof}
Without loss of generality, take~$E$ to be a long edge of~$P$ achieving the maximal height~$h$ among all long edges of~$P$. By Proposition~\ref{prop:heights_fixed_3_sym}, the only edges of~$P$ which have height~$h' \le h$ are in~$\set{E, GE, G^2 E}$; otherwise, the primitive inner normal of such an edge would be a lattice point of~$h P^* \cap M$. Thus, we may conclude that~$E$,~$GE$, and~$G^2 E$ are the only long edges of~$P$ and that all other edges of~$P$ must have height strictly greater than~$h$. It now follows that if an R-singularity of~$P$ has height~$h$, then it must be on one of the long edges of~$P$.
\end{proof}
Finally, we give the following lemma specific to~$3$-symmetric polygons which will be of use in the following subsection.
\begin{lemma}\label{lemma:height_bound_3sym}
Let~$P \subset \NQ$ be a~$3$-symmetric polygon with non-empty basket~$\Bc$. Suppose~$E$ is an edge of~$P$ with height~$h$. Then~$h \ge 2$.
\end{lemma}
\begin{proof}
Suppose towards a contradiction that~$h = 1$. Then~$E$ is a long edge. Since~$\Bc \neq \varnothing$, case (ii) of Lemma~\ref{lemma:ord3_det} applies. Therefore,~$P$ is contained in a reflexive triangle~$Q$ with~$X_Q = \PP^2 / \ZZ_3$. In particular,~$P$ itself is reflexive. This is a contradiction, since it implies that~$\Bc = \varnothing$.
\end{proof}
\subsection{The proof for symmetric Fano polygons}\label{subsec:proof_sym}
In this subsection, the aim is to prove Theorem~\ref{thm:sym_unique}. We first prove it for the case when the basket of R-singularities is empty. Then, we prove the non-empty basket case.

So, we begin by proving that all symmetric Fano polygons are minimal.
\begin{lemma}\label{lemma:P_minimal}
Let~$P \subset \NQ$ be a symmetric Fano polygon. Then~$P$ is minimal.
\end{lemma}
\begin{proof}
If~$P$ is centrally symmetric, then minimality of~$P$ is clear. It remains to treat the case when~$P$ is~$3$-symmetric. If~$P$ has no long edges, then we are done. Otherwise, let~$E$ be a long edge of~$P$ with length~$\ell$ and height~$h$. In order to demonstrate minimality of~$P$, we need to prove that there is a point~$\vec{x}$ of~$P$ satisfying~$u(\vec{x}) \ge h$.

We can apply Lemma~\ref{lemma:ord3_det}. In case (i), it's straightforward to see that~$P$ is minimal. We now consider case (ii). Without loss of generality,~$u = (0,1)^t$ is the inner normal to~$E$ and~$(1,0)^t$ is the inner normal to~$GE$, where~$G \in \Aut(P)$ is an element of order~$3$.
Thus,~$G = \left(\begin{smallmatrix}
 -1 & -1 \\
 1 & 0
\end{smallmatrix}\right)$. We may write~$E = \conv\set{(a,-h), (a+\ell,-h)}$ for some integer~$a$. Since~$E$ is long, the point~$(a+h,-h)$ is in~$E$.

Consider the points~$G \cdot (a,-h) = (-h,h-a)$ and~$G^2 \cdot (a+h,-h) = (-a,h+a)$, which belong to~$P$. If we evaluate them at~$u$, we obtain~$h-a$ and~$h+a$, respectively. At least one of these is greater than or equal to~$h$. Thus, we are done.
\end{proof}
Previous work of~\cite{minimality} already contains the answer for symmetric Fano polygons whose baskets of R-singularities are empty.
\begin{proposition}\label{prop:empty_basket}
Let~$P, Q \subset \NQ$ be symmetric Fano polygons with empty basket~$\Bc$. If~$P$ and~$Q$ are mutation-equivalent, then~$P$ is isomorphic to~$Q$.
\end{proposition}
\begin{proof}
In~\cite[Theorem 5.4]{minimality}, all minimal Fano polygons with~$\Bc = \varnothing$ are classified up to isomorphism. Since symmetric Fano polygons are minimal by Lemma~\ref{lemma:P_minimal}, the ones with empty basket all appear in their list.

We display all~$6$ symmetric Fano polygons with empty baskets in Figure~\ref{fig:empty_basket_syms}. They each have a different number of primitive T-singularities. So, none are mutation-equivalent to each other. Therefore, no two non-isomorphic symmetric Fano polygons are mutation-equivalent to each other.

\begin{figure}[h]
\centering
\begin{subfigure}{.2\textwidth}
 \begin{tikzpicture}[scale=0.7][line cap=round,line join=round,x=1cm,y=1cm]
 \begin{axis}[
 x=1cm,y=1cm,
 axis lines=middle,
 ymajorgrids=true,
 xmajorgrids=true,
 xmin=-2,
 xmax=2,
 ymin=-2,
 ymax=2
 ]
 
 \fill[line width=2pt,color=blue,fill=blue,fill opacity=0.05] (0,1) -- (1,0) -- (-1,-1) -- cycle;
 
 \draw [line width=2pt,color=black] (0,1)-- (1,0);
 \draw [line width=2pt,color=black] (1,0)-- (-1,-1);
 \draw [line width=2pt,color=black] (-1,-1)-- (0,1);
 
 \begin{scriptsize}
 \draw [fill=black] (0,1) circle (2.5pt);
 \draw [fill=black] (1,0) circle (2.5pt);
 \draw [fill=black] (-1,-1) circle (2.5pt);
 \end{scriptsize}
 \end{axis}
 \end{tikzpicture}
 \subcaption{$n = 3$.}
\end{subfigure}
\begin{subfigure}{.2\textwidth}
 \begin{tikzpicture}[scale=0.7][line cap=round,line join=round,x=1cm,y=1cm]
 \begin{axis}[
 x=1cm,y=1cm,
 axis lines=middle,
 ymajorgrids=true,
 xmajorgrids=true,
 xmin=-2,
 xmax=2,
 ymin=-2,
 ymax=2
 ]
 
 \fill[line width=2pt,color=blue,fill=blue,fill opacity=0.05] (0,1) -- (1,0) -- (0,-1) -- (-1,0) -- cycle;
 
 \draw [line width=2pt,color=black] (0,1)-- (1,0);
 \draw [line width=2pt,color=black] (1,0)-- (0,-1);
 \draw [line width=2pt,color=black] (0,-1)-- (-1,0);
 \draw [line width=2pt,color=black] (-1,0)-- (0,1);
 
 \begin{scriptsize}
 \draw [fill=black] (0,1) circle (2.5pt);
 \draw [fill=black] (1,0) circle (2.5pt);
 \draw [fill=black] (0,-1) circle (2.5pt);
 \draw [fill=black] (-1,0) circle (2.5pt);
 \end{scriptsize}
 \end{axis}
 \end{tikzpicture}
 \subcaption{$n = 4$.}
\end{subfigure}
\begin{subfigure}{.2\textwidth}
 \begin{tikzpicture}[scale=0.7][line cap=round,line join=round,x=1cm,y=1cm]
 \begin{axis}[
 x=1cm,y=1cm,
 axis lines=middle,
 ymajorgrids=true,
 xmajorgrids=true,
 xmin=-2,
 xmax=2,
 ymin=-2,
 ymax=2,
 ]
 
 \fill[line width=2pt,color=blue,fill=blue,fill opacity=0.05] (0,1) -- (1,1) -- (1,0) -- (0,-1) -- (-1,-1) -- (-1,0) -- cycle;
 
 \draw [line width=2pt,color=black] (0,1) -- (1,1);
 \draw [line width=2pt,color=black] (1,1) -- (1,0);
 \draw [line width=2pt,color=black] (1,0) -- (0,-1);
 \draw [line width=2pt,color=black] (0,-1) -- (-1,-1);
 \draw [line width=2pt,color=black] (-1,-1) -- (-1,0);
 \draw [line width=2pt,color=black] (-1,0) -- (0,1);
 
 \begin{scriptsize}
 \draw [fill=black] (0,1) circle (2.5pt);
 \draw [fill=black] (1,1) circle (2.5pt);
 \draw [fill=black] (1,0) circle (2.5pt);
 \draw [fill=black] (0,-1) circle (2.5pt);
 \draw [fill=black] (-1,-1) circle (2.5pt);
 \draw [fill=black] (-1,0) circle (2.5pt);
 \end{scriptsize}
 \end{axis}
 \end{tikzpicture}
 \subcaption{$n = 6$.}
\end{subfigure}
\begin{subfigure}{.2\textwidth}
 \begin{tikzpicture}[scale=0.7][line cap=round,line join=round,x=1cm,y=1cm]
 \begin{axis}[
 x=1cm,y=1cm,
 axis lines=middle,
 ymajorgrids=true,
 xmajorgrids=true,
 xmin=-2,
 xmax=2,
 ymin=-2,
 ymax=2,
 ]
 
 \fill[line width=2pt,color=blue,fill=blue,fill opacity=0.05] (-1,-1) -- (-1,1) -- (1,1) -- (1,-1) -- cycle;
 
 \draw [line width=2pt,color=black] (-1,-1) -- (-1,1);
 \draw [line width=2pt,color=black] (-1,1) -- (1,1);
 \draw [line width=2pt,color=black] (1,1) -- (1,-1);
 \draw [line width=2pt,color=black] (1,-1) -- (-1,-1);
 
 \begin{scriptsize}
 \draw [fill=black] (-1,-1) circle (2.5pt);
 \draw [fill=black] (-1,1) circle (2.5pt);
 \draw [fill=black] (1,1) circle (2.5pt);
 \draw [fill=black] (1,-1) circle (2.5pt);
 \end{scriptsize}
 \end{axis}
 \end{tikzpicture}
 \subcaption{$n = 8$.}
\end{subfigure}
\begin{subfigure}{.3\textwidth}
 \begin{tikzpicture}[scale=0.65][line cap=round,line join=round,x=1cm,y=1cm]
 \begin{axis}[
 x=1cm,y=1cm,
 axis lines=middle,
 ymajorgrids=true,
 xmajorgrids=true,
 xmin=-2,
 xmax=3,
 ymin=-2,
 ymax=3,
 ]
 
 \fill[line width=2pt,color=blue,fill=blue,fill opacity=0.05] (-1,-1) -- (2,-1) -- (-1,2) -- cycle;
 
 \draw [line width=2pt,color=black] (-1,-1) -- (2,-1);
 \draw [line width=2pt,color=black] (2,-1) -- (-1,2);
 \draw [line width=2pt,color=black] (-1,2) -- (-1,-1);
 
 \begin{scriptsize}
 \draw [fill=black] (-1,-1) circle (2.5pt);
 \draw [fill=black] (2,-1) circle (2.5pt);
 \draw [fill=black] (-1,2) circle (2.5pt);
 \end{scriptsize}
 \end{axis}
 \end{tikzpicture}
 \subcaption{$n = 9$.}
\end{subfigure}
\begin{subfigure}{.3\textwidth}
 \begin{tikzpicture}[scale=0.7][line cap=round,line join=round,x=1cm,y=1cm]
 \begin{axis}[
 x=1cm,y=1cm,
 axis lines=middle,
 ymajorgrids=true,
 xmajorgrids=true,
 xmin=-3,
 xmax=3,
 ymin=-2,
 ymax=2,
 xtick={-3,...,3},
 ytick={-2,...,2}
 ]
 
 \fill[line width=2pt,color=blue,fill=blue,fill opacity=0.05] (-2,-1) -- (-2,1) -- (2,1) -- (2,-1) -- cycle;
 
 \draw [line width=2pt,color=black] (-2,-1) -- (-2,1);
 \draw [line width=2pt,color=black] (-2,1) -- (2,1);
 \draw [line width=2pt,color=black] (2,1) -- (2,-1);
 \draw [line width=2pt,color=black] (2,-1) -- (-2,-1);
 
 \begin{scriptsize}
 \draw [fill=black] (-2,-1) circle (2.5pt);
 \draw [fill=black] (-2,1) circle (2.5pt);
 \draw [fill=black] (2,1) circle (2.5pt);
 \draw [fill=black] (2,-1) circle (2.5pt);
 \end{scriptsize}
 \end{axis}
 \end{tikzpicture}
 \subcaption{$n = 10$.}
\end{subfigure}
\caption{The six symmetric Fano polygons with~$n$ primitive T-singularities and empty basket~$\Bc = \varnothing$.}
\label{fig:empty_basket_syms}
\end{figure}
\end{proof}
So, we can restrict our attention to symmetric Fano polygons which have a non-empty basket~$\Bc$ of R-singularities. We prove the remainder of Theorem~\ref{thm:sym_unique} in two big steps. In the first step, we prove that two symmetric Fano polygons with the same edge data are isomorphic (Proposition~\ref{prop:one_cs_per_extended_basket}). In the second step, we prove that if two symmetric Fano polygons are mutation-equivalent, then their edge data is the same (Proposition~\ref{prop:ESC_same}). Thus, Theorem~\ref{thm:sym_unique} would then follow.

In order to complete the first step, we need the following two lemmas.
\begin{lemma}\label{lemma:sym_det_fixed}
Let~$H = \left(\begin{smallmatrix}
 1 & a \\
 0 & r
\end{smallmatrix}\right)$ be a Hermite normal form satisfying~$\gcd(a,r) = 1$. Fix an edge~$E$ with primitive vertices~$\vec{v}_0$ and~$\vec{v}_1$ in~$N$, which are ordered anticlockwise. Let~$u_E \in M$ be the primitive inner normal of~$E$ and let~$h_E$ be the height of~$E$. Let~$\Fc$ be the set of all edges~$F$ which satisfy the following conditions:
\begin{itemize}
\item[(i)] $F$ has Hermite normal form~$H$;
\item[(ii)] The first vertex of~$F$ is the second vertex~$\vec{v}_1$ of~$E$;
\item[(iii)] The second vertex~$\vec{v}_2$ of~$F$ satisfies~$-h < u(\vec{v}_2) \le (r-1)h$.
\end{itemize}
Then,~$|\Fc| \le 1$.
\end{lemma}
\begin{proof}
Since the vertices are primitive, we may assume without loss of generality that~$\vec{v}_0 = (b,-s)$ and~$\vec{v}_1 = (1,0)$, for some coprime integers~$b,s > 0$. If~$\Fc = \varnothing$, then we are done. Otherwise, let~$F, F' \in \Fc$ so that, by (ii),~$F$ has vertices~$\vec{v}_1$ and~$\vec{v}_2$ and~$F'$ has vertices~$\vec{v}_2'$, ordered anticlockwise. We aim to show that~$F = F'$.

Write~$\vec{v}_2 = (x,y)$ and~$\vec{v}_2' = (x',y')$. By (i), there exists some~$U, U' \in \SL_2(\ZZ)$ such that~$F = UH$ and~$F' = U'H$. It follows that~$y = y' = r$ and~$x \congr x' \congr a \mod r$. Thus,~$x' = x + dr$ for some~$d \in \ZZ$.

It remains to show that~$d=0$. To do this, we apply condition (iii). An important observation is that the half-open strip~$S \coloneqq \setcond{\vec{x} \in \NQ}{-h < u_E(\vec{x}) \le (r-1)h}$ is the Minkowski sum of the half-open segment~$\setcond{(x,0)}{1-r \le x < 1}$ and a line with non-horizontal slope. From this, we can see that at each fixed~$y$-coordinate, there are at most~$r$ lattice points in~$S$. The vertices~$\vec{v}_2$ and~$\vec{v}_2'$ both lie in~$S$ and share the same~$y$-coordinate; consequently,~$|x-x'| < r$, i.e.\ $|dr| < r$. Thus,~$d=0$. Therefore,~$\vec{v}_2 = \vec{v}_2'$, and so~$F = F'$. It now follows that~$|\Fc| \le 1$.

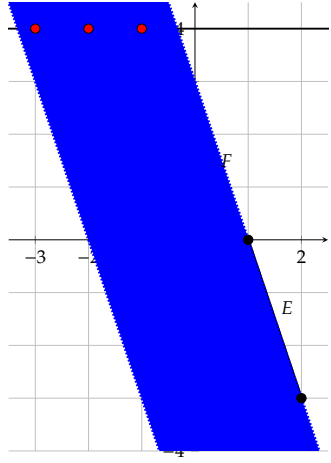
\begin{figure}[ht]
\centering
\begin{tikzpicture}[scale=0.7][line cap=round,line join=round,>=triangle 45,x=1cm,y=1cm]
 \begin{axis}[
 x=1cm,y=1cm,
 axis lines=middle,
 ymajorgrids=true,
 xmajorgrids=true,
 xmin=-3.5,
 xmax=2.5,
 ymin=-4,
 ymax=4.5,
 xtick={-3,...,2},
 ytick={-4,-3,...,4}]

 \draw [line width=1.5pt] (2,-3) -- (1,0);
 \draw [line width=1.5pt] (1,0) -- (-1,4);
 \draw [line width=1pt,domain=-4:5] plot(\x,4);
 
 \draw[line width=1.5pt,dash pattern=on 1pt off 1pt,color=blue,fill=blue,fill opacity=0.1]
 (-4,6) -- (0,-6) -- (3,-6) -- (-1,6);
 
 \begin{scriptsize}
 \draw [fill=black] (2,-3) circle (2.5pt);
 \draw [fill=black] (1,0) circle (2.5pt);
 \draw[color=black] (1.74,-1.28) node {$E$};
 \draw[color=black] (0.6,1.5) node {$F$};
 
 \draw [fill=red] (-3,4) circle (2.5pt);
 \draw [fill=red] (-2,4) circle (2.5pt);
 \draw [fill=red] (-1,4) circle (2.5pt);
 \end{scriptsize}
 \end{axis}
\end{tikzpicture}
\caption{An example to demonstrate Lemma~\ref{lemma:sym_det_fixed}. Here,~$E = \conv\set{(2,-3), (1,0)}$ and~$H = \left(\begin{smallmatrix}
 1 & 3 \\
 0 & 4
\end{smallmatrix}\right)$. The red points are the three choices for the second vertex~$\vec{v}_2$ of~$F$; the only valid choice for~$\vec{v}_2$ is~$(-1,4)$.}
\label{fig:one_vertex_choice}
\end{figure}
\end{proof}
This next lemma generalises a behaviour shared by centrally symmetric polygons and~$3$-symmetric polygons.
\begin{lemma}\label{lemma:sym_height_bound}
Let~$P \subset \NQ$ be a symmetric polygon and~$E$ be an edge of~$P$. Then,~$P$ is contained in the strip supported by~$E$ and~$-2E$.
\end{lemma}
\begin{proof}
Let~$S$ be the strip supported by~$E$ and~$-2E$, i.e.\ $S \coloneqq \setcond{\vec{x} \in \NQ}{-h \le u(\vec{x}) \le 2h}$, where~$u$ is the primitive inner normal to~$E$ and~$h$ is the height of~$E$. If~$P$ is centrally symmetric, then~$-E$ is an edge of~$P$. So, by convexity,~$P \subseteq S$. Otherwise,~$P$ is~$3$-symmetric. Now, as in the proof of Lemma~\ref{lemma:ord3_det}, consider the triangle~$Q \coloneqq \setcond{\vec{x} \in \NQ}{(uG^k)(\vec{x}) \ge -h, \text{ for } k=0,1,2}$. It has weights~$(1,1,1)$. Thus,~$Q$ is isomorphic to the triangle with vertices~$(-a,-h)$,~$(\ell-a,-h)$, and~$(2a-\ell,2h)$, for some~$a, \ell > 0$. So,~$Q$ is contained in the strip~$S$ and, since~$P \subseteq Q$, the result follows.
\end{proof}
We can now use the above two lemmas to complete the first step, which was to prove the following statement.
\begin{proposition}\label{prop:one_cs_per_extended_basket}
Let~$P,Q \subset \NQ$ be symmetric Fano polygons with non-empty basket of R-singularities. Suppose they have the same edge data, i.e.\ $\Ec(P) = \Ec(Q)$. Then~$P$ is isomorphic to~$Q$.
\end{proposition}
\begin{proof}
Let~$\Ec(P) = \Ec(Q) = \set{H_1, H_2, \ldots, H_m}$. To get the desired result, we aim to repeatedly apply Lemma~\ref{lemma:sym_det_fixed}. In order to do so, we must show that all the hypotheses of the lemma hold. Let~$E$ and~$F$ be two consecutive edges of~$P$, ordered anticlockwise. Let~$H$ be the Hermite normal form of~$F$. Since~$P$ is Fano, the vertices of~$F$ are primitive; hence,~$H$ satisfies the requirements of Lemma~\ref{lemma:sym_det_fixed}. We next show that~$F$ belongs to the set~$\Fc$ of Lemma~\ref{lemma:sym_det_fixed}.

Condition (i) is satisfied by the definition of~$H$. Since~$E$ is adjacent to~$F$, condition (ii) is satisfied. It remains to show condition (iii). Consider the second vertex~$\vec{v}_2$ of~$F$. Since~$\vec{v}_2$ does not lie on~$E$, the first inequality~$u_E(\vec{v}_2) > -h$ holds. In order to prove the second inequality, we must use the fact that~$P$ is symmetric. By Lemma~\ref{lemma:sym_height_bound}, we have~$u_E(\vec{v}_2) \le 2h$. Now, it remains to show that~$2h \le (r-1)h$, i.e.\ that~$r \ge 3$.

We note that~$r = \ell_F h_F$. Since~$P$ is symmetric and has empty basket, it follows from Lemma~\ref{lemma:height_bound} and Lemma~\ref{lemma:height_bound_3sym} that~$h_F \ge 2$. If~$h_F \ge 3$, then~$r \ge 3$ and we are done. Otherwise,~$F$ has height~$h_F = 2$. Since the vertices of~$P$ are primitive, the length~$\ell_F$ of~$F$ must be even; hence,~$\ell_F \ge 2$. So,~$r \ge 4$, and we are done.

We may now prove the statement of the proposition. Label anticlockwise the edges of~$P$ and~$Q$ as~$E_1, E_2, \ldots, E_m$ and~$E_1', E_2', \ldots, E_m'$, respectively, so that~$E_i$ and~$E_i'$ have Hermite normal form~$H_i$. So, there exists some~$U \in \SL_2(\ZZ)$ such that~$UE_1' = E_1$. The next edge of~$UQ$ is~$UE_2'$. We may apply Lemma~\ref{lemma:sym_det_fixed} to obtain~$UE_2' = E_2$. By induction, we obtain that~$UE_i' = E_i$ for all~$i = 1, 2, \ldots, m$. Therefore,~$Q = UP$, and we are done.
\end{proof}
We now move onto the final step in order to prove Theorem~\ref{thm:sym_unique}. Before we proceed with the final proposition of the section, we require the following small result.
\begin{lemma}\label{lemma:one_edge_below}
Let~$P \subset \NQ$ be a polygon. Let~$u \in M$ be primitive and let~$n \ge 2$ be an integer. Suppose there are exactly~$n$ vertices~$\vec{v}$ of~$P$ which satisfy~$u(\vec{v}) \le 0$. If~$P$ is centrally symmetric, then~$P$ has at most~$2n$ vertices. If~$P$ is~$3$-symmetric, then~$P$ has at most~$3n$ vertices.
\end{lemma}
\begin{proof}
Without loss of generality, we may take~$u = (0,1)^t$. Let~$\vec{v}_1, \vec{v}_2, \ldots, \vec{v}_n$ be the~$n$ vertices of~$P$ satisfying~$u(\vec{v}_i) \le 0$, i.e.\ they lie on or below the~$x$-axis. We assume that the vertices are ordered anticlockwise.

First, suppose that~$P$ is centrally symmetric. Then~$P$ has either~$0$ or~$2$ vertices on the~$x$-axis and either~$n$ or~$n-2$ vertices strictly below it, respectively. By the central symmetry of~$P$, it has either~$n$ or~$n-2$ vertices strictly above it, respectively. So,~$P$ has either~$2n$ or~$2n-2$ vertices in total. In particular,~$P$ has at most~$2n$ vertices.

Instead, suppose that~$P$ is~$3$-symmetric. Let~$G \in \Aut(P)$ be an element of order~$3$. Let~$\vec{v}_0$ be the vertex of~$P$ immediately before~$\vec{v}_1$ and let~$\vec{v}_{n+1}$ be the vertex of~$P$ immediately after~$\vec{v}_n$. By assumption, these two vertices lie strictly above the~$x$-axis. Consider the triangle~$T = \conv\set{\vec{v}_0, G\vec{v}_0, G^2\vec{v}_0}$ and the line~$L$ passing through~$\vec{v}_0$ and the origin~$\vec{0}$. Without loss of generality,~$G\vec{v}_0$ lies below~$L$ and~$G^2\vec{v}_0$ lies above~$L$. By convexity of~$P$,~$G\vec{v}_0$ in fact lies below the~$x$-axis. Thus,~$G\vec{v}_0 = \vec{v}_i$ for some~$i = 1,2,\ldots,n$. This means that the vertices of~$P$ are~$\vec{v}_0, \vec{v}_1, \ldots, \vec{v}_{i-1}, G\vec{v}_0, G\vec{v}_1, \ldots, G\vec{v}_{i-1}, G^2\vec{v}_0, G^2\vec{v}_1, \ldots, G^2\vec{v}_{i-1}$. Therefore,~$P$ has~$3i \le 3n$ vertices.

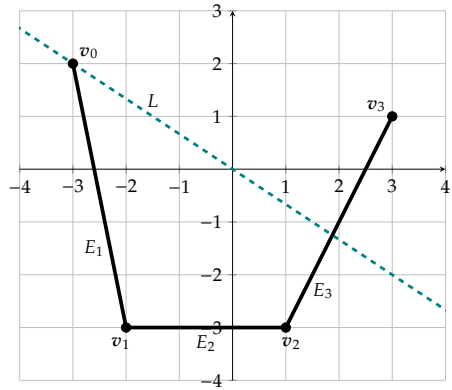
\begin{figure}[h]
\centering
\begin{tikzpicture}[scale=0.7][line cap=round,line join=round,x=1cm,y=1cm]
 \begin{axis}[
 x=1cm,y=1cm,
 axis lines=middle,
 ymajorgrids=true,
 xmajorgrids=true,
 xmin=-4,
 xmax=4,
 ymin=-4,
 ymax=3,
 xtick={-5,...,7},
 ytick={-4,...,3}]

 \draw[line width=2pt] (-3,2) -- (-2,-3); 
 \draw[line width=2pt] (-2,-3) -- (1,-3); 
 \draw[line width=2pt] (1,-3) -- (3,1); 
 \draw[dashed, color=teal, line width=1.5] (-6,4) -- (6,-4); 
 
 \begin{scriptsize}
 \draw [fill=black] (-3,2) circle (2.5pt);
 \draw[color=black] (-2.7, 2.2) node {$\vec{v}_0$};
 
 \draw [fill=black] (-2,-3) circle (2.5pt);
 \draw[color=black] (-2.1, -3.3) node {$\vec{v}_1$};
 
 \draw [fill=black] (1,-3) circle (2.5pt);
 \draw[color=black] (1.1, -3.3) node {$\vec{v}_2$};

 \draw [fill=black] (3,1) circle (2.5pt);
 \draw[color=black] (2.7, 1.2) node {$\vec{v}_3$};

 \draw[color=black] (-2.6, -1.5) node {$E_1$};
 \draw[color=black] (-0.5, -3.3) node {$E_2$};
 \draw[color=black] (1.7, -2.3) node {$E_3$};
 \draw[color=black] (-1.5, 1.3) node {$L$};
 \end{scriptsize}
 \end{axis}
\end{tikzpicture}
\caption{An example with~$n=2$ and~$P$ is assumed to be~$3$-symmetric. The points~$\vec{v}_1$ and~$\vec{v}_2$ are the only two vertices of~$P$ below the~$x$-axis. One of these must also be a vertex of the triangle~$T$.}
\label{fig:sym_edge_below}
\end{figure}
\end{proof}
We may now complete the final step, which was to prove the following statement.
\begin{proposition}\label{prop:ESC_same}
Let~$P$ and~$Q$ be symmetric Fano polygons with non-empty baskets. If they are mutation-equivalent, then their edge data is equal.
\end{proposition}
\begin{proof}
If~$P$ and~$Q$ are isomorphic, then their edge data is equal and we are done. So, let us assume that~$P$ is not isomorphic to~$Q$. Thus,~$P$ must have at least one long edge; otherwise,~$P$ has no mutations and~$P \isom Q$ -- a contradiction. If~$P$ is centrally symmetric then, by Proposition~\ref{prop:one_long_edge_cs},~$P$ is isomorphic to~$Q$ -- a contradiction. Thus, either~$P$ is centrally symmetric with two pairs of long edges or~$P$ is~$3$-symmetric with one triple of long edges. The same holds for~$Q$. Applying Proposition~\ref{prop:heights_fixed_cs} and Proposition~\ref{prop:heights_fixed_3_sym}, we see that~$P$ and~$Q$ must either both be centrally symmetric or both be~$3$-symmetric; otherwise, the Ehrhart series~$\Ehr_{P^*}(t)$ and~$\Ehr_{Q^*}(t)$ would differ -- this is a contradiction because the Ehrhart series of the dual polygon is a mutation-invariant, and~$P$ and~$Q$ are assumed to be mutation-equivalent. Moreover, by Corollary~\ref{cor:num_edges_cs} and Corollary~\ref{cor:num_edges_3sym}, we see that the number of edges is fixed, i.e.\ $|\Ec(P)| = |\Ec(Q)|$.

So, let's write~$\Ec(P) = \set{H_1, H_2, \ldots, H_m} \times g$ and~$\Ec(Q) = \set{H'_1, H'_2, \ldots, H'_m} \times g$, for some integer~$m \ge 1$ and where~$g = 2$ if~$P$ is centrally symmetric and~$g = 3$ if~$P$ is~$3$-symmetric (see Remark~\ref{remark:times_notation} for notation). The (directed) basket of~$P$ and~$Q$ can be written~$\Bc = \set{S_1, S_2, \ldots, S_b} \times g$. We may insist that~$\Ec(P)$ is~\emph{aligned} with~$\Bc$, i.e.\ that~$\res(H_1) = S_1$ and~$\res(H_{i_j}) = S_j$, for~$1 = i_1 < i_2 < \cdots < i_b \le m$. We may also insist that~$\Ec(Q)$ is aligned with~$\Bc$.

The strategy of the remainder of the proof is by induction on the elements of~$\Ec(P)$. In particular, we want to show that~$H_1 = H'_1$ and that if~$H_n = H'_n$ for~$1 \le n \le k$, where~$1 \le k < m$, then~$H_{k+1} = H'_{k+1}$. Given these two statements, it immediately follows that~$\Ec(P) = \Ec(Q)$.

\underline{Base case:}
We first prove that~$H_1$ and~$H'_1$ coincide. Suppose towards a contradiction that all the edges of~$P$ are long. If~$P$ is centrally symmetric, then~$P$ is isomorphic to the box with vertices~$\pm(h_1, h_2)$ and~$\pm(h_1, -h_2)$, for some integers~$2 \le h_1 \le h_2$. By Corollary~\ref{cor:num_edges_cs},~$Q$ is also a quadrilateral. By Lemma~\ref{lemma:det_one},~$Q$ is isomorphic to the box with vertices~$\pm(h'_1, h'_2)$ and~$\pm(h'_1, -h'_2)$, for some integers~$2 \le h'_1 \le h'_2$. By Proposition~\ref{prop:heights_fixed_cs}, its edges have heights~$h_1$ and~$h_2$, i.e.\ $h_1 = h'_1$ and~$h_2 = h'_2$. Thus,~$P \isom Q$, a contradiction. Otherwise,~$P$ is~$3$-symmetric. So, since all its edges are long, it is isomorphic to the triangle with vertices~$(-h,-h)$,~$(2h,-h)$, and~$(-h,2h)$. Since~$P$ is Fano, its vertices must be primitive. Thus,~$h=1$, which implies that~$\Bc = \varnothing$, another contradiction.

So, we may conclude that there is an edge of~$P$ which is short. Since the edge data~$\Ec(P)$ is cyclically ordered, we may relabel its elements so that~$H_1$ represents a short edge; thus,~$H_1 = \res(H_1) = S_1$. Now, this implies that~$S_1$ has height strictly greater than the heights of the long edges of~$P$ and~$Q$. Since~$\res(H'_1) = S_1$, it follows that~$H'_1$ is also short. Thus,~$H_1 = H'_1 = S_1$.

\underline{Inductive step:}
Let~$1 \le k < m$ and suppose that~$H_n = H'_n$ for all~$1 \le n \le k$. Let~$1 \le j \le b$ be the index of the previous R-singularity~$S_j$ in~$\set{\res(H_1), \ldots, \res(H_k)}$, i.e.\ $i_j \le k$ and either~$j=b$ or~$k < i_{j+1}$. Consider~$H_{k+1}$ and~$H'_{k+1}$. We split into three cases: (i) both are short; (ii) both are long; and (iii) one is short and one is long.

Case (i). If both forms represent short edges, then~$H_{k+1} = S_{j+1}$ and~$H'_{k+1} = S_{j+1}$, i.e.\ they coincide.

Case (ii). Now suppose that~$H_{k+1}$ and~$H'_{k+1}$ both represent long edges. Fix the edge of~$P$ and~$Q$ represented by~$H_k$ so that its second vertex is~$(1,0)$. Now, the primitive inner normal vectors of the next edges of~$P$ and~$Q$, which are represented by~$H_{k+1}$ and~$H'_{k+1}$, respectively, are forced to be the same; otherwise, we would violate Proposition~\ref{prop:heights_fixed_cs} or Proposition~\ref{prop:heights_fixed_3_sym}. In order to show that~$H_{k+1}$ and~$H'_{k+1}$ coincide, we show that the next edges have the same length.

Since the edges represented by~$H_{k+1}$ and~$H'_{k+1}$ share a vertex and have the same primitive inner normal, they must have the same height. Hence, they have the same number of primitive T-singularities. If they have the same residue, then we are done. Otherwise, due to the directed basket being a mutation invariant, we must have that, without loss of generality,~$H_{k+1}$ has a residue while~$H'_{k+1}$ has no residue. Now, consider the Hermite normal form~$H'_{k+2}$ in~$\Ec(Q)$. ~$H'_{k+2}$ must represent a long edge with the same residue as~$H_{k+1}$. The long edges represented by~$H'_{k+1}$ and~$H'_{k+2}$ have the same height~$h$. By Lemma~\ref{lemma:det_one} and Lemma~\ref{lemma:ord3_det}, we can transform~$Q$ so that the common vertex of these long edges is~$(h,h)$. Since~$Q$ is Fano, its vertices must be primitive. Thus,~$h=1$. But this is a contradiction; for example, the residue of~$H_{k+1}$ is now empty. So,~$H_{k+1} = H'_{k+1}$.

Case (iii). Finally, we suppose that one of~$H_{k+1}$ and~$H'_{k+1}$ is short and the other is long. Without loss of generality, we may assume that the former is long and the latter is short. So,~$H_{k+1} = T$ represents a~\emph{pure} long edge, i.e.\ $\res(H_{k+1}) = \varnothing$, and~$H'_{k+1} = S_{j+1}$ is short. Now, consider the next edge represented in~$\Ec(P)$. There are two subcases: either (a)~$H_{k+2} = S_{j+1}$ represents a short edge or (b)~$H_{k+2} = T'$ represents a pure long edge. We aim to show that both subcases are impossible to achieve.

Subcase (a): As in Figure~\ref{fig:subcase_a}, transform~$P$ so that the edge~$E$ represented by~$T$ has inner normal~$(0,1)^t$. Denote by~$F_1$ and~$F_2$ the edges adjacent to~$E$ which are represented by~$H_k$ and~$S_{j+1}$, respectively. Let~$\vec{v}_0, \vec{v}_1, \vec{v}_2, \vec{v}_3 \in N$ be the vertices which form the edges~$F_1$,~$E$, and~$F_2$, ordered anticlockwise. Consider the shear~$A$ which maps~$\vec{v}_1$ to~$\vec{v}_2$. The edge~$AF_1$ is now adjacent to~$F_2$. We may also transform~$Q$ so that its edge represented by~$S_{j+1}$ is also~$F_2$. The aim is to show that~$AF_1$ must be the edge of~$Q$ represented by~$H_k$. If that is the case then, by convexity,~$Q$ is contained in the half-space~$\set{y \ge -h}$, i.e.\ $(0,1)^t \in hQ^*$, where~$h$ is the height of~$E$. But now, this contradicts Proposition~\ref{prop:heights_fixed_cs} and Proposition~\ref{prop:heights_fixed_3_sym}, since~$(0,1)^t$ is not an inner normal vector to an edge of~$Q$. So, we can conclude that subcase (a) is impossible to achieve. It remains to show that~$AF_1$ and~$F_2$ are indeed edges of~$Q$. To do this, we want to apply Lemma~\ref{lemma:sym_det_fixed}. In order to apply it, we must show its condition (iii) is satisfied, i.e.\ we want to show that~$A\vec{v}_0$ satisfies the chain of inequalities:~$-h_{F_2} < u_{F_2}(A\vec{v}_0) \le 2h_{F_2}$.

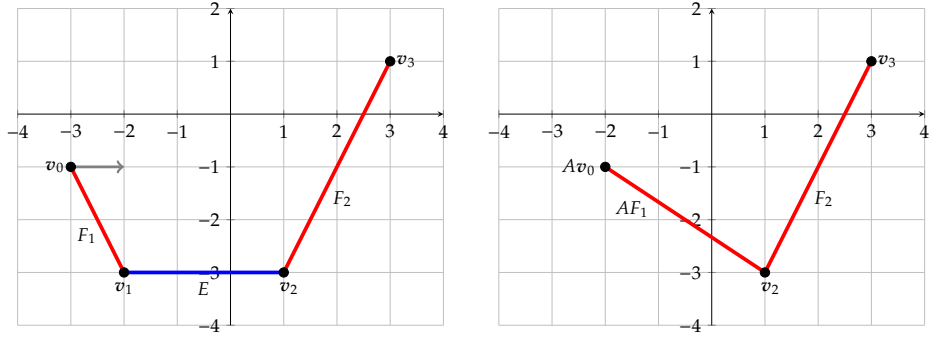
\begin{figure}[h]
\centering
\begin{subfigure}{.4\textwidth}
 \begin{tikzpicture}[scale=0.7][line cap=round,line join=round,x=1cm,y=1cm]
 \begin{axis}[
 x=1cm,y=1cm,
 axis lines=middle,
 ymajorgrids=true,
 xmajorgrids=true,
 xmin=-4,
 xmax=4,
 ymin=-4,
 ymax=2,
 xtick={-5,...,7},
 ytick={-4,...,2}]

 \draw [->,line width=1.5pt, color=gray] (-3,-1) -- (-2,-1); 
 \draw [line width=2pt, color=red] (-3,-1) -- (-2,-3); 
 \draw [line width=2pt, color=blue] (-2,-3) -- (1,-3); 
 \draw [line width=2pt, color=red] (1,-3) -- (3,1); 
 
 \begin{scriptsize}
 \draw [fill=black] (-3,-1) circle (2.5pt);
 \draw [fill=black] (-2,-3) circle (2.5pt);
 \draw [fill=black] (1,-3) circle (2.5pt);
 \draw [fill=black] (3,1) circle (2.5pt);

 \draw[color=black] (-3.3, -1) node {$\vec{v}_0$};
 \draw[color=black] (-2, -3.3) node {$\vec{v}_1$};
 \draw[color=black] (1.1, -3.3) node {$\vec{v}_2$};
 \draw[color=black] (3.3, 1) node {$\vec{v}_3$};

 \draw[color=black] (-2.7, -2.3) node {$F_1$};
 \draw[color=black] (-0.5, -3.3) node {$E$};
 \draw[color=black] (2.1, -1.6) node {$F_2$};
 \end{scriptsize}
 \end{axis}
 \end{tikzpicture}
\end{subfigure}
\begin{subfigure}{.4\textwidth}
 \begin{tikzpicture}[scale=0.7][line cap=round,line join=round,x=1cm,y=1cm]
 \begin{axis}[
 x=1cm,y=1cm,
 axis lines=middle,
 ymajorgrids=true,
 xmajorgrids=true,
 xmin=-4,
 xmax=4,
 ymin=-4,
 ymax=2,
 xtick={-5,...,7},
 ytick={-4,...,2}]

 \draw [line width=2pt, color=red] (-2,-1) -- (1,-3); 
 \draw [line width=2pt, color=red] (1,-3) -- (3,1); 
 
 \begin{scriptsize}
 \draw [fill=black] (-2,-1) circle (2.5pt);
 \draw [fill=black] (1,-3) circle (2.5pt);
 \draw [fill=black] (3,1) circle (2.5pt);
 
 \draw[color=black] (-2.5, -1) node {$A\vec{v}_0$};
 \draw[color=black] (1.1, -3.3) node {$\vec{v}_2$};
 \draw[color=black] (3.3, 1) node {$\vec{v}_3$};
 
 \draw[color=black] (-1.5, -1.8) node {$AF_1$};
 \draw[color=black] (2.1, -1.6) node {$F_2$};
 \end{scriptsize}
 \end{axis}
 \end{tikzpicture}
\end{subfigure}
\caption{The situation of subcase (a) in the proof of Proposition~\ref{prop:ESC_same}. In this example,~$E$ has length equal to its height; thus,~$A$ is the shear~$\left(\begin{smallmatrix}
 1 & -1 \\
 0 & 1
\end{smallmatrix}\right)$.}
\label{fig:subcase_a}
\end{figure}

By the assumption of subcase (a), if~$P$ is centrally symmetric, it has at least~$6$ edges and if~$P$ is~$3$-symmetric, it has at least~$9$ edges. So, in order to not contradict Lemma~\ref{lemma:one_edge_below}, it follows that at least one of~$\vec{v}_0$ and~$\vec{v}_3$ lies on or below the~$x$-axis. Without loss of generality,~$\vec{v}_0$ lies on or below the~$x$-axis.

We may now write~$\vec{v}_0 = \lambda\vec{v}_1 - (\mu,0)$, for some~$0 \le \lambda < 1$ and~$\mu > 0$. Thus,~$A\vec{v}_0 = \lambda\vec{v}_2 - (\mu,0)$. So,~$u_{F_2}(A\vec{v}_0) = -\lambda h_{F_2} -\mu u_{F_2}(1,0)$. Since~$u_{F_2}(1,0) < 0$, the first inequality holds. On the other hand,~$A\vec{v}_0 = \vec{v}_0 + (x,0)$ for some~$x \in \ZZ$. Since~$\vec{v}_0$ lies on or below the~$x$-axis, it follows that~$x \ge 0$. Furthermore, by Lemma~\ref{lemma:sym_height_bound}, we have~$u_{F_2}(\vec{v}_0) \le 2h_{F_2}$. Thus, the second inequality follows. As set out above, this is now enough to conclude that subcase (a) cannot occur.

Subcase (b): We first note that~$P$ (and~$Q$) must be centrally symmetric, since there are two long edges between~$H_1$ and~$H_m$. Thus, it follows that $H_{k+3}$ must be short, i.e. $H_{k+3} = S_{j+1}$. As in Figure~\ref{fig:subcase_b}, transform~$P$ so that the edge~$E$ represented by~$T$ has inner normal~$(0,1)^t$ and the edge~$E'$ represented by~$T'$ has inner normal~$(-1,0)^t$. Denote by~$F_1$ and~$F_2$ the edges adjacent to~$E$ which are represented by~$H_k$ and~$S_{j+1}$, respectively. Let~$\vec{v}_0, \vec{v}_1, \ldots, \vec{v}_4 \in N$ be the vertices which form the edges~$F_1$,~$E$,~$E'$, and~$F_2$, ordered anticlockwise.

\begin{figure}[h]
\centering
\begin{subfigure}{.4\textwidth}
 \begin{tikzpicture}[scale=0.7][line cap=round,line join=round,x=1cm,y=1cm]
 \begin{axis}[
 x=1cm,y=1cm,
 axis lines=middle,
 ymajorgrids=true,
 xmajorgrids=true,
 xmin=-3,
 xmax=4,
 ymin=-5.,
 ymax=4,
 xtick={-3,...,4},
 ytick={-5,...,4}]

 \draw [->,line width=1.5pt, color=gray] (-2,-3) -- (1,-3); 
 \draw [->,line width=1.5pt, color=gray] (1,-3) -- (1,-2); 
 \draw [line width=2pt, color=blue] (-1, -4) -- (3, -4);
 \draw [line width=2pt, color=blue] (3, -4) -- (3, -1);
 \draw [line width=2pt, color=red] (-1, -4) -- (-2, -3);
 \draw [line width=2pt, color=red] (3, -1) -- (2, 3);
 
 \begin{scriptsize}
 \draw [fill=black] (-1, -4) circle (2.5pt);
 \draw [fill=black] (3, -4) circle (2.5pt);
 \draw [fill=black] (3, -1) circle (2.5pt);
 \draw [fill=black] (-2, -3) circle (2.5pt);
 \draw [fill=black] (2, 3) circle (2.5pt);
 
 \draw[color=black] (-1.9, -2.7) node {$\vec{v}_0$};
 \draw[color=black] (-1, -4.3) node {$\vec{v}_1$};
 \draw[color=black] (2.9, -4.3) node {$\vec{v}_2$};
 \draw[color=black] (3.3, -0.9) node {$\vec{v}_3$};
 \draw[color=black] (1.7, 3) node {$\vec{v}_4$};
 
 \draw[color=black] (-1.7, -3.6) node {$F_1$};
 \draw[color=black] (1, -4.2) node {$E$};
 \draw[color=black] (3.2, -2.5) node {$E'$};
 \draw[color=black] (2.7, 1.2) node {$F_2$};
 \end{scriptsize}
 \end{axis}
 \end{tikzpicture}
\end{subfigure}
\begin{subfigure}{.4\textwidth}
 \begin{tikzpicture}[scale=0.7][line cap=round,line join=round,x=1cm,y=1cm]
 \begin{axis}[
 x=1cm,y=1cm,
 axis lines=middle,
 ymajorgrids=true,
 xmajorgrids=true,
 xmin=-3,
 xmax=4,
 ymin=-5.,
 ymax=4,
 xtick={-3,...,4},
 ytick={-5,...,4}]
 
 \draw [line width=2pt, color=red] (1,-2) -- (3, -1);
 \draw [line width=2pt, color=red] (3, -1) -- (2, 3);
 
 \begin{scriptsize}
 \draw [fill=black] (1, -2) circle (2.5pt);
 \draw [fill=black] (3, -1) circle (2.5pt);
 \draw [fill=black] (2, 3) circle (2.5pt);
 
 \draw[color=black] (0.8, -2.4) node {$A'A\vec{v}_0$};
 \draw[color=black] (3.3, -0.9) node {$\vec{v}_3$};
 \draw[color=black] (1.7, 3) node {$\vec{v}_4$};
 
 \draw[color=black] (2.5, -1.7) node {$A'AF_1$};
 \draw[color=black] (2.7, 1.2) node {$F_2$};
 \end{scriptsize}
 \end{axis}
 \end{tikzpicture}
\end{subfigure}
\caption{The situation of subcase (b) in the proof of Proposition~\ref{prop:ESC_same}. In this example,~$E$ and~$E'$ both have lengths equal to their heights; thus,~$A = \left(\begin{smallmatrix}
 1 & -1 \\
 0 & 1
\end{smallmatrix}\right)$ and~$A' = \left(\begin{smallmatrix}
 1 & 0 \\
 1 & 1
\end{smallmatrix}\right)$.}
\label{fig:subcase_b}
\end{figure}
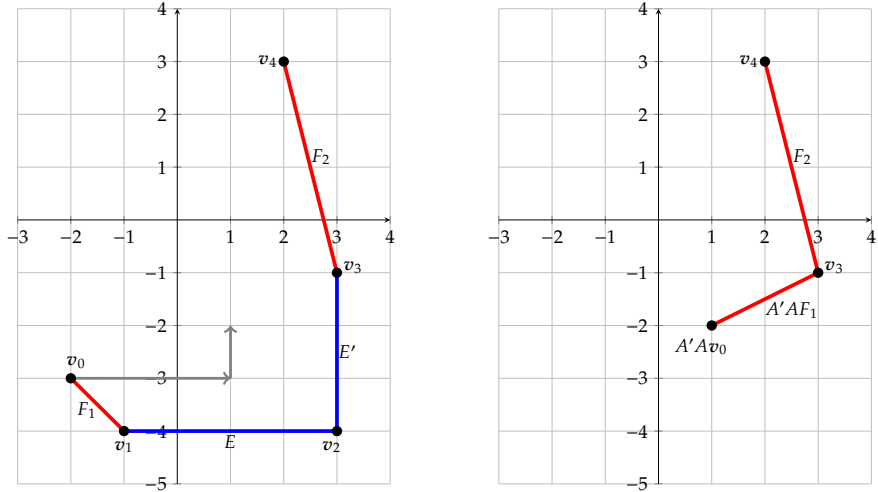

Consider the half space~$\Hc = \setcond{(x,y) \in \NQ}{x \ge y}$.
It is clear that~$\vec{v}_2 \in \Hc$. By assumption of subcase (b), note that~$P$ has at least~$8$ vertices. So, in order to not contradict Lemma~\ref{lemma:one_edge_below}, it follows that at least one of~$\vec{v}_0$ and~$\vec{v}_4$ lies belongs to~$\Hc$. Without loss of generality, we may assume that~$\vec{v}_0, \vec{v}_1, \vec{v}_2, \vec{v}_3 \in \Hc$.

Let~$A$ and~$A'$ be the shears which map~$\vec{v}_1$ to~$\vec{v}_2$ and~$\vec{v}_2$ to~$\vec{v}_3$, respectively. We want to show that~$A'A\vec{v}_0$ satisfies the inequalities~$-h_{F_2} < u_{F_2}(A'A\vec{v}_0) \le h_{F_2}$. Note that~$\vec{v}_0 = \lambda\vec{v}_1 - (\mu,0)$, for some~$0 \le \lambda < 1$ and~$\mu > 0$. Thus,~$A'A\vec{v}_0 = \lambda\vec{v}_3 - (\mu, n'\mu)$, where~$n'$ is the number of primitive singularities of~$E'$. Since both~$u_{F_2}(1,0)$ and~$u_{F_2}(0,1)$ are strictly negative, the first inequality holds. On the other hand,~$A'A\vec{v}_0 = \vec{v}_0 + (x,y)$, for some~$x, y \in \ZZ$. Since~$\vec{v}_0$ lies in~$\Hc$, it follows that~$x, y \ge 0$. Furthermore, by the central symmetry of $P$, we have~$u_{F_2}(\vec{v}_0) \le h_{F_2}$; thus, the second inequality holds.

Thus, we may apply Lemma~\ref{lemma:sym_det_fixed} and transform~$Q$ so that its edges represented by~$H_k$ and~$S_{j+1}$ are~$A'AF_1$ and~$F_2$. But now, we still have~$(-1,0)^t \in h'Q^*$, where~$h'$ is the height of the long edge~$E'$. But by Proposition~\ref{prop:heights_fixed_cs}, the non-origin lattice points in~$h'Q^*$ must be primitive inner normals to long edges; thus, we reach a contradiction.
\end{proof}
We are now ready to prove the main theorem of this section.
\begin{proof}[Proof of Theorem~\ref{thm:sym_unique}]
Let~$P$ and~$Q$ be two symmetric Fano polygons which are mutation-equivalent. By Proposition~\ref{prop:empty_basket}, we may assume that~$P$ and~$Q$ have non-empty basket~$\Bc$. By Proposition~\ref{prop:ESC_same}, we know that their edge data is equal. Finally, we apply Proposition~\ref{prop:one_cs_per_extended_basket} and we may conclude that~$P$ and~$Q$ are isomorphic.
\end{proof}
\section{The behaviour of other K{\"a}hler--Einstein polygons under mutation}\label{section:KE-triangles}
In this section, we aim to complete the proof of Theorem~\ref{thm:main}. Previously, in~\S\ref{section:most_one_cs_fano}, we proved that there is at most one symmetric Fano polygon in each mutation-equivalence class. Equivalently, we proved that if two symmetric Fano polygons are mutation-equivalent, then they are isomorphic. So, in order to complete the proof of Theorem~\ref{thm:main}, we need to prove that (a) if two \KE\ Fano triangles are mutation-equivalent, then they are isomorphic and (b) if a symmetric Fano polygon is mutation-equivalent to a \KE\ triangle, then they are isomorphic. We prove (a) in Proposition~\ref{prop:triangle_triangle}. Since the arguments for centrally symmetric and~$3$-symmetric polygons will be different, we further split the proof of (b) into Proposition~\ref{prop:triangle_cs} and Proposition~\ref{prop:triangle_3sym}.
\begin{proposition}\label{prop:triangle_triangle}
Let~$P, Q \subset \NQ$ be \KE\ Fano triangles. Suppose that~$P$ and~$Q$ are mutation-equivalent. Then~$P$ is isomorphic to~$Q$.
\end{proposition}
Before we do the proof of Proposition~\ref{prop:triangle_triangle}, we require the following lemma.
\begin{lemma}\label{lemma:KE_odd}
Let~$P \subset \NQ$ be a \KE\ Fano triangle of index~$k$. Then~$k$ is odd.
\end{lemma}
\begin{proof}
By Lemma~\ref{lemma:KE_triangle_verts}, we may apply an appropriate transformation so that~$P$ has vertices~$(-k,a-1)$, $(k,-a)$, and~$(0,1)$, for some integers~$a$ and~$k$. Suppose towards a contradiction that~$k$ is even. Now, since~$2$ divides either~$a-1$ or~$a$, it follows that either~$(-k, a-1)$ or~$(k, -a)$ is not primitive. So,~$P$ has a non-primitive vertex, which contradicts that~$P$ is Fano.
\end{proof}
We may now proceed with the proof.
\begin{proof}[Proof of Proposition~\ref{prop:triangle_triangle}]
By Lemma~\ref{lemma:KE_triangle_verts},~$P$ is isomorphic to a triangle with primitive vertices~$(-k, a-1)$, $(k,-a)$, and~$(0,1)$, for some integers~$a,k \ge 1$. Similarly,~$Q$ is isomorphic to a triangle with primitive vertices~$(-k', a'-1)$,~$(k',-a')$, and~$(0,1)$, for some integers~$a',k' \ge 1$. The dual polygons~$P^*$ and~$Q^*$ have normalised volumes~$9/k$ and~$9/k'$, respectively. Since~$P$ and~$Q$ are mutation-equivalent, these volumes coincide; thus,~$k=k'$. It now remains to show that~$a \congr a' \mod k$.

Label the edges of~$P$ as~$E_0, E_1, E_2$. We can write~$k = \ell_i h_i$, where~$\ell_i$ is the length of~$E_i$ and~$h_i$ is its height. We obtain the following system of linear congruences:
\begin{equation}\label{eq:ke_triangle_congr}
a \congr -1 \mod \ell_0, \ \ \ \ \ \ \
a \congr 2 \mod \ell_1, \ \ \ \ \ \ \
2a \congr 1 \mod \ell_2.
\end{equation}
Thus, if~$m$ divides two of the lengths~$\ell_i$ and~$\ell_j$ for~$i \neq j$, then~$m$ must divide~$3$.

If~$P$ does not have at least two long edges, then~$Q$ must be isomorphic to~$P$. So, we assume otherwise. Thus, without loss of generality,~$E_0$ and~$E_1$ are long. This implies that~$\ell_0, \ell_1 \geq \sqrt{k}$.

If~$\gcd(\ell_0, \ell_1) = 1$, then~$\ell_0 \ell_1$ divides~$k$. Since~$\ell_0 \ell_1 \ge k$, we obtain~$\ell_0 = \sqrt{k} = \ell_1$. Thus,~$k=1$. But now, there is only one such triangle with index~$k=1$, so~$P$ and~$Q$ are isomorphic; in particular,~$X_P = X_Q = \PP^2$.

Now consider the case~$\gcd(\ell_0, \ell_1) = 3$. Here, we have~$\ell_0 = 3\widehat{\ell}_0$ and~$\ell_1 = 3\widehat{\ell}_1$, where~$\widehat{\ell}_0$ and~$\widehat{\ell}_1$ are coprime. So,~$k = 3\widehat{\ell}_0 \widehat{\ell}_1 \widehat{k} \le 9\widehat{\ell}_0 \widehat{\ell}_1$. So,~$\widehat{k} \leq 3$. By Lemma~\ref{lemma:KE_odd}, we have that~$\widehat{k} \neq 2$.
This leaves two possibilities for~$\widehat{k}$.

If~$\widehat{k}=3$, then~$\ell_0 = \ell_1 = \sqrt{k}$. Thus,~$\widehat{\ell_0} = \widehat{\ell_1} = 1$,~$\ell_0 = \ell_1 = 3$, and~$k=9$. Up to isomorphism, there is only one Fano triangle with weights~$(1,1,1)$ and index~$k=9$. In fact, this is the triangle $P'$ appearing in Example~\ref{eg:nonsym_KE_tri}. Therefore,~$P$ and~$Q$ must be isomorphic.

If~$\widehat{k} = 1$, then~$k=3\widehat{\ell}_0 \widehat{\ell}_1$. Since~$\widehat{\ell}_0$ and~$\widehat{\ell}_1$ are coprime, we may assume without loss of generality that~$\widehat{\ell}_1$ is not divisible by~$3$. Now, consider the system~\eqref{eq:ke_triangle_congr}. We obtain~$a \congr -1 \mod 3\widehat{\ell}_0$ and~$a \congr 2 \mod \widehat{\ell_1}$. Since~$3\widehat{\ell}_0$ and~$\widehat{\ell}_1$ are coprime, we may apply the Chinese Remainder theorem. We obtain a unique solution for~$a$ modulo~$k$, which determines~$P$ and~$Q$ up to isomorphism. Thus,~$P$ and~$Q$ are isomorphic.
\end{proof}
Now that we have shown the triangle/triangle case, it remains to show the triangle/symmetric case. It will be useful to have the following lemma as it allows us to verify the case when the basket of R-singularities is empty.
\begin{lemma}\label{lemma:KE_triangle_minimal}
Let~$P \subset \NQ$ be a \KE\ Fano triangle. Then~$P$ is minimal.
\end{lemma}
\begin{proof}
By Lemma~\ref{lemma:KE_triangle_weights},~$P$ has weights~$(1,1,1)$. Let~$\vec{v}_0, \vec{v}_1, \vec{v}_2$ be the vertices of~$P$ and~$u_0, u_1, u_2$ be the vertices of~$P^*$, ordered so that~$u_i(\vec{v}_j) = -1$ if and only if~$i \neq j$. In order to show minimality of~$P$, it's enough to show that~$u_i(\vec{v}_i) \ge 1$, for all~$i=0,1,2$. To do this, we simply rearrange the identity~$u_i(\vec{v}_0 + \vec{v}_1 + \vec{v}_2) = 0$. We obtain~$u_i(\vec{v}_i) = 2 \ge 1$, and so we are done.
\end{proof}
We can now prove the remaining two propositions.
\begin{proposition}\label{prop:triangle_cs}
No \KE\ Fano triangle is mutation-equivalent to a centrally symmetric Fano polygon.
\end{proposition}
\begin{proof}
Let~$P$ be a centrally symmetric Fano polygon and~$Q$ be a \KE\ triangle. Suppose towards a contradiction that~$P$ and~$Q$ are mutation-equivalent. Due to the central symmetry of~$P$, the number of R-singularities in the basket must be even. Since~$Q$ has three edges, there are at most three R-singularities. Thus, there are either zero or two R-singularities in the basket.

First suppose that there are no R-singularities in the basket. By Lemma~\ref{lemma:KE_triangle_minimal}, the triangle~$Q$ will be minimal. Thus, the triangle~$Q$ appears in the list of~\cite{minimality}. The centrally symmetric~$P$ is also minimal and appears in that list. We see that there are no examples in the same mutation-equivalence class.

Otherwise, there are two copies of the same R-singularity in~$\Bc$. Consider the edges of the triangle~$Q$. Two of the edges~$F_1, F_2$ host R-singularities and one edge~$F_0$ is a pure T-singularity. It follows from Lemma~\ref{lemma:KE_triangle_weights} that the determinant of each edge is~$k$, for some positive integer $k$. So, the edges~$F_1$ and~$F_2$ have the same number~$m \ge 0$ of primitive T-singularities. Now, since~$P$ is centrally symmetric, it has an even number~$2n$ of primitive T-singularities. Thus, the pure edge~$F_0$ of~$Q$ has~$2n - 2m$ primitive T-singularities. Further, its determinant~$k$ must now be divisible by~$2$. But this contradicts Lemma~\ref{lemma:KE_odd}.
\end{proof}
\begin{proposition}\label{prop:triangle_3sym}
Let~$P \subset \NQ$ be a \KE\ Fano triangle and~$Q \subset \NQ$ be a~$3$-symmetric Fano polygon. Suppose that~$P$ is mutation-equivalent to~$Q$. Then~$P \isom Q$.
\end{proposition}
\begin{proof}
First, we know that~$P$ and~$Q$ are minimal. So, if~$\Bc = \varnothing$, then~$P$ and~$Q$ must appear in the list of minimal polygons in Theorem 5.4 of~\cite{minimality}. We see that there is at most one \KE\ triangle or~$3$-symmetric polygon for each number~$n$ of primitive T-singularities, which is a mutation-invariant. Therefore,~$P \isom Q$.

Now, we have~$\Bc \neq \varnothing$. The number of R-singularities of~$P$ is a multiple of~$3$. Further, the number of R-singularities of~$Q$ is at most~$3$. Since~$P$ and~$Q$ are mutation-equivalent, they have the same number of R-singularities, which must be~$3$. Because an edge can have at most one R-singularity and~$Q$ is a triangle, each of its edges has an R-singularity. Due to the symmetry of order~$3$ on~$P$, these three R-singularities must be the same; in particular, they have the same height~$h$. Therefore, each edge of~$Q$ has the same height~$h$. Since each edge of~$Q$ must also be long, we may apply Lemma~\ref{lemma:same_height_KE}. So, we must have~$X_Q = \PP^2$ or~$X_Q = \PP^2 / \ZZ_3$. In either case,~$Q$ now has no R-singularities, which is a contradiction.
\end{proof}
Now we simply put all these results together to derive the main theorem.
\begin{proof}[Proof of Theorem~\ref{thm:main}]
The result straightforwardly follows from Theorem~\ref{thm:sym_unique}, Proposition~\ref{prop:triangle_triangle}, Proposition~\ref{prop:triangle_cs}, and Proposition~\ref{prop:triangle_3sym}.
\end{proof}
\section{Discussion of K{\"a}hler--Einstein polygons}\label{section:KE_polygons}\label{section:countereg}
Given Conjecture~\ref{conj:hk}, it would follow that there is at most one \KE\ Fano polygon in each mutation-equivalence class. However, in the first part of this section, we provide a counterexample to the conjecture. In particular, we show the existence of a \KE\ Fano polygon which is neither symmetric nor a triangle. This is enough to disprove the conjecture.

The second part of this section is dedicated to describing all \KE\ Fano quadrilaterals. In particular, we describe how quadrilaterals with barycentre as the origin must be constrained.

In the final part of this section, we construct a family of \KE Fano hexagons which are not minimal.
\subsection{A K{\"a}hler--Einstein Fano quadrilateral}\label{sec:proof_of_countereg}
So, as stated above, we will prove that the quadrilateral in Proposition~\ref{prop:countereg} is indeed \KE\ and non-symmetric.
\begin{proof}[Proof of Proposition~\ref{prop:countereg}]
Since all of the vertices of~$P$ are primitive lattice points and the origin is contained in its strict interior,~$P$ is Fano. Using Magma~\cite{magma}, we find that~$P$ has a trivial automorphism group; thus,~$P$ is not symmetric. As~$P$ is clearly not a triangle, it only remains to show that the barycentre of~$P^*$ is the origin.

Again using Magma, we compute the dual~$P^*$ of~$P$.
\[
P^* = \conv\set{\frac{1}{3149}(1,-28)^t, \frac{1}{1739}(151,2)^t, \frac{1}{481}(-31,4)^t, \frac{1}{871}(-43,-2)^t} \subset \MQ.
\]
Label its vertices as~$u_1, u_2, u_3, u_4$, respectively. Note that these are adjacent and given in anticlockwise order. To compute the barycentre, we subdivide~$P^*$ into two triangles~$T_1 = \conv\set{u_1, u_2, u_3}$ and~$T_2 = \conv\set{u_3, u_4, u_1}$. Let~$b_1$ and~$b_2$ be the barycentres of~$T_1$ and~$T_2$, respectively. Then,~$b_1 = \frac{1}{1514669}(11461, 290)^t$ and~$b_2 = \frac{1}{1514669}(-57305, -1450)^t$. The volumes of~$T_1$ and~$T_2$ are~$\frac{3240}{1514669}$ and~$\frac{648}{1514669}$, respectively. Thus, the barycentre of~$P^*$ is~$\frac{1}{\Vol(P^*)}(\Vol(T_1) b_1 + \Vol(T_2) b_2) = (0,0)^t$.

So,~$P$ is a \KE\ Fano polygon which is not symmetric and not a triangle. Its existence thus disproves Conjecture~\ref{conj:hk}.
\end{proof}
This means that it is still a wide-open question whether there is at most one \KE\ polygon per mutation-equivalence class.
\begin{remark}
We can ask whether the polygon~$P$ in the proof of Proposition~\ref{prop:countereg} is mutation-equivalent to any other \KE\ polygon. Since the edges of~$P$ have lengths 1, 2, 7, 2 and heights 3149, 871, 481, 1739, respectively, it follows that none of its edges are long. Thus, the mutation-equivalence class consists only of~$P$. So, in this case,~$P$ is not mutation-equivalent to any other other \KE\ polygon.
\end{remark}
Having this example of a non-symmetric \KE\ quadrilateral allows us to test questions regarding barycentric transformations. These transformations were introduced in~\cite{barycentric_transformations} as a way to measure how close a polytope is to being \KE. We recall the definition, specialised to dimension two.
\begin{definition}[\!{\cite[Lemma~2.2]{barycentric_transformations}}]
Let~$P$ be a Fano polygon with vertices~$\vec{v}_1, \vec{v}_2, \ldots, \vec{v}_m$ written in anticlockwise order. Then the~\emph{barycentric transformation}~$B(P)$ of~$P$ is defined as
\[
B(P) \coloneqq \conv\set{\frac{\vec{v}_{12}}{\gcd(\vec{v}_{12})}, \frac{\vec{v}_{23}}{\gcd(\vec{v}_{23})}, \ldots, \frac{\vec{v}_{m1}}{\gcd(\vec{v}_{m1})}},
\]
where~$\vec{v}_{ij} = \vec{v}_i + \vec{v}_j$. We also write~$B^1(P) \coloneqq B(P)$ and~$B^{k+1}(P) \coloneqq B^k(B(P))$, for integer~$k \ge 1$.
\end{definition}
Note that the vertices of~$B(P)$ are guaranteed to be primitive. However, it isn't certain that the origin is contained in the strict interior of~$B(P)$. So, the barycentric transformation of a Fano polygon might not be Fano itself.
\begin{definition}[\!{\cite[Definition~2.4]{barycentric_transformations}}]
A Fano polygon~$P$ is said to be of~\emph{type~$B_k$} if~$B^k(P)$ is also Fano. $P$ is of~\emph{strict} type~$B_k$ if~$P$ is of type~$B_k$ but not of type~$B_{k+1}$. $P$ is of \emph{type~$B_\infty$} if~$B^k(P)$ is Fano for all integer~$k \ge 1$.
\end{definition}
In~\cite{barycentric_transformations}, they expect that all smooth \KE\ (Fano) polytopes are of type~$B_\infty$. They show that this expectation holds in all dimensions less than~$9$, except for possibly one four-dimensional polytope~\cite[Theorem~1.3]{barycentric_transformations}. In two dimensions, they drop the smooth condition and prove that all \KE\ polygons are of type~$B_1$~\cite[Theorem~1.6]{barycentric_transformations}. They question whether \KE\ Fano polygons are all of type~$B_\infty$. If so, they also question whether the \KE\ property is preserved under barycentric transformation. We answer both questions in the negative.
\begin{example}\label{eg:barytra}
We compute iterated barycentric transformations of the polygon~$P$ from Proposition~\ref{prop:countereg} to determine its strict type. Consider~$B(P) = \conv\set{(10, -163), (17, 70), (2, 225), (-11, -39)}$. It contains the origin in its strict interior, so it is Fano. However, the barycentre of its dual is not the origin; thus,~$B(P)$ is not \KE. Next,~$B^2(P) = \conv\set{(9, -31), (19, 295), (-3, 62), (-1, -202)}$. Since it still contains the origin in its strict interior, it is Fano. But now,~$B^3(P) = \conv\set{(7, 66), (16, 357), (-1,-35), (8,-233)}$. This polygon does not contain the origin in its strict interior; thus,~$B^3(P)$ is not Fano. We can conclude that~$P$ is of strict type~$B_2$.
\end{example}
\subsection{The weight systems of quadrilaterals with barycentre as the origin}
\begin{proposition}\label{prop:KE_quad_weights}
Let~$P \subset \NQ$ be a \KE\ quadrilateral. Then its dual polygon~$P^*$ is a quadrilateral with weight system
\[
W^* = \begin{pmatrix}
 \lambda_1 & \lambda_2 & 0 & \lambda_4 \\
 0 & \mu_2 & \mu_3 & \mu_4
\end{pmatrix},
\]
for some positive integers~$\lambda_1, \lambda_2, \lambda_4, \mu_2, \mu_4$ and negative integer~$\mu_3$ which satisfy
\begin{align}
\label{eq:weights1} (\lambda_1 - \lambda_2) f \cdot \mu_3^2 &= (\mu_3 - \mu_2) g \cdot \lambda_1^2 \\
\label{eq:weights2} (\lambda_1 - \lambda_4) f \cdot \mu_3^2 &= (\mu_3 - \mu_4) g \cdot \lambda_1^2,
\end{align}
where~$f = \lambda_1 + \lambda_2 + \lambda_4$ and~$g = \mu_2 + \mu_3 + \mu_4$.
\end{proposition}
Before we prove Proposition~\ref{prop:KE_quad_weights}, we must first prove the following lemma about triangles.
\begin{lemma}\label{lemma:triangle_vols}
Let~$T \subset \NQ$ be a triangle with vertices~$\vec{v}_0, \vec{v}_1, \vec{v}_2$ and weights~$(\lambda_0, \lambda_1, \lambda_2)$, so that~$\sum_{i=0}^2 \lambda_i \vec{v}_i = \vec{0}$. Label the edges of~$T$ as~$E_0, E_1, E_2$ so that~$\vec{v}_i \not\in E_i$ for~$i=0,1,2$. Suppose that the weights are~\emph{reduced}, i.e.\ $\gcd(\lambda_0, \lambda_1, \lambda_2) = 1$. Then, for all~$i=0,1,2$, the volume of~$\Delta_{E_i}$ is~$k|\lambda_i|$, for some positive~$k \in \QQ$ not depending on~$i$.
\end{lemma}
\begin{proof}
As in the remark under Definition~2.3 in~\cite{conrads}, this result also follows in the non-integral case using Cramer's rule.
\end{proof}
We are now ready to prove the main result of this subsection.
\begin{proof}[Proof of Proposition~\ref{prop:KE_quad_weights}]
Let~$\ell = \gcd(\lambda_1, \lambda_2, \lambda_4)$ and~$m = \gcd(\mu_2, \mu_3, \mu_4)$. Consider~\eqref{eq:weights1}. Since~$\ell$ divides~$f$, we see that~$\ell^2$ divides the left-hand side and the right-hand side. Further, since~$m$ divides~$g$, we see that~$m^2$ also divides both the left-hand and right-hand sides. The analogous statement holds for~\eqref{eq:weights2}. Thus, we may divide the~$\lambda_i$ by~$\ell$ and the~$\mu_j$ by~$m$. Hence, we may assume without loss of generality that~$\ell = m = 1$.

Let~$u_1, u_2, u_3, u_4 \in \MQ$ be the vertices of~$P^*$, given in anticlockwise order and aligned so that~$W^* (u_1, u_2, u_3, u_4) = (0, 0)$. Then we obtain two triangles~$T = \conv\set{u_1, u_2, u_4}$ and~$T' = \conv\set{u_2, u_3, u_4}$ which partition~$P^*$. The barycentres of~$T$ and~$T'$ are~$b = \frac13(u_1 + u_2 + u_4)$ and~$b' = \frac13(u_2 + u_3 + u_4)$, respectively.

Now,~$T$ has weights~$(\lambda_1, \lambda_2, \lambda_4)$ and~$T'$ has weights~$(\mu_2, \mu_3, \mu_4)$. Let~$E_i = \conv\setcond{u_n}{n \neq i}$ and~$E'_j = \conv\setcond{u_n}{n \neq j}$, for~$i=1,2,4$ and~$j=2,3,4$. By Lemma~\ref{lemma:triangle_vols}, the triangles~$\Delta_{E_i}$ have volumes~$k_T \lambda_i$, for~$i=1,2,4$, and the triangles~$\Delta_{E'_i}$ have volumes~$k_{T'} |\mu_j|$, for~$j=2,3,4$, where~$k_T, k_{T'} \in \QQ$ are positive. But now,~$\Delta_{E_1} = \Delta_{E'_3}$. In particular,~$\Vol(\Delta_{E_1}) = \Vol(\Delta_{E'_3})$. It follows that the ratio between~$k_1$ and~$k_2$ is~$(-\mu_3 : \lambda_1)$. Therefore, the ratio between the volumes of~$T$ and~$T'$ is~$(-f\mu_3 : \lambda_1 g)$.

Since the barycentre of~$P^*$ is~$0$, it follows that~$\Vol(T)b + \Vol(T')b' = 0$. Plugging in our expressions for the volumes and barycentres of~$T$ and~$T'$, we obtain that~$-f\mu_3 \cdot u_1 + (\lambda_1 g - f\mu_3) \cdot u_2 + \lambda_1 g \cdot u_3 + (\lambda_1 g - f\mu_3) \cdot u_4 = 0$. Equivalently,~$(-f\mu_3, \lambda_1 g - f\mu_3, \lambda_1 g, \lambda_1 g - f\mu_3)$ is a weight for~$P^*$. So, it must be expressable as a~$\QQ$-linear combination of the rows of~$W^*$:
\begin{equation}\label{eq:lin_comb}
(-f\mu_3, \lambda_1 g - f\mu_3, \lambda_1 g, \lambda_1 g - f\mu_3) = a \cdot (\lambda_1, \lambda_2, 0, \lambda_4) + b \cdot (0, \mu_2, \mu_3, \mu_4),
\end{equation}
for some~$a,b \in \QQ$. Comparing the first coordinates in~\eqref{eq:lin_comb} gives~$a = -f\mu_3 / \lambda_1$ and comparing the third coordinates in~\eqref{eq:lin_comb} gives~$b = \lambda_1 g / \mu_3$. Then, obtaining~\eqref{eq:weights1} is simply a matter of comparing the second coordinates in~\eqref{eq:lin_comb}, substituting in the expressions of~$a$ and~$b$, and rearranging. Equation~\eqref{eq:weights2} is obtained similarly from comparing the fourth coordinates in~\eqref{eq:lin_comb}.
\end{proof}
\begin{remark}\label{remark:KE_quad}
We can create a nice parameterisation of the weight matrix in Proposition~\ref{prop:KE_quad_weights}. First, we may assume that each row of the weight matrix is reduced, i.e.\ $\gcd(\lambda_1, \lambda_2, \lambda_4) = 1$ and~$\gcd(\mu_2, \mu_3, \mu_4) = 1$. It follows from Equations~\eqref{eq:weights1} and~\eqref{eq:weights2} that~$\lambda_1^2$ divides~$f \cdot \mu_3^2$ and~$\mu_3^2$ divides~$g \cdot \lambda_1^2$. So, let~$a = \gcd(\lambda_1, \mu_3)$ and let~$\ell$ and~$m$ be coprime integers such that~$\lambda_1 = a\ell$ and~$\mu_3 = am$. Now, let~$d = \gcd(f, g)$ and let~$r$ and~$s$ be coprime integers such that~$f = d\ell^2 s$ and~$g = dm^2 r$. Finally, let~$b = \gcd(\lambda_1 - \lambda_2, \mu_3 - \mu_2)$ and~$c = \gcd(\lambda_1 - \lambda_4, \mu_3 - \mu_4)$. We can now express all~$\lambda_i$ and~$\mu_j$ in terms of the parameters~$a, b, c, d, \ell, m, r, s$. The weight matrix looks like the following.
\[
W^* = \begin{pmatrix}
 a\ell & a\ell + br & 0 & a\ell + cr \\
 0 & am + bs & am & am + cs
\end{pmatrix}.
\]

Using the identities~$f = \lambda_1 + \lambda_2 + \lambda_4$ and~$g = \mu_2 + \mu_3 + \mu_4$, we eventually obtain the equations
\begin{align}
\label{eq:param1} b + c &= -d\ell m \\
\label{eq:param2} 3a &= d(mr + \ell s).
\end{align}
We can use this parameterisation to search for further examples of non-symmetric \KE\ Fano polygons.
\end{remark}
\begin{example}
Let us fix~$a = 48$. By~\eqref{eq:param2},~$d$ divides~$144$. We choose~$d = 36$. By~\eqref{eq:param1},~$b+c$ must be divisible by~$36$. Choose~$b = 19$ and~$c = 17$. Referring back to~\eqref{eq:param1}, we see that~$36 = -36\ell m$. Thus,~$\ell = 1$ and~$m = -1$. Finally,~\eqref{eq:param2} gives~$4 = -r + s$. We may let~$r = 1$ and~$s = 5$.

This gives the following weight matrix
\[
W^* = \begin{pmatrix}
 48 & 67 & 0 & 65 \\
 0 & 47 & -48 & 37
\end{pmatrix},
\]
which corresponds to a quadrilateral with barycentre zero. Taking the dual polygon and restricting to the lattice spanned by its vertices, we obtain the \KE\ Fano quadrilateral from Proposition~\ref{prop:countereg}.
\end{example}
\subsection{Non-symmetric K{\"a}hler--Einstein polygons coming from symmetric polygons}\label{subsec:nonsym_from_3sym}
In the previous subsection, we showed one way to construct non-symmetric \KE\ polygons which are not triangles. By construction, all previous examples were quadrilaterals. In this subsection, we will construct a different type of non-symmetric \KE\ polygon. We first illustrate the main idea of the construction with the following example.
\begin{example}\label{eg:nonsym_KE_tri}
Take the triangle~$P$ with vertices~$(-1,-1)$,~$(1,0)$, and~$(0,1)$. This is \KE\ and symmetric, and has corresponding toric variety~$X_P = \PP^2$. Consider~$P$ with respect to the lattice~$N + \frac{1}{9}(1,2)\ZZ$. This is isomorphic to the triangle~$P'$ with vertices~$(-5,1)$,~$(1,-2)$, and~$(4,1)$, with respect to~$N$, and whose corresponding toric variety is~$X_{P'} = \PP^2 / \ZZ_9$. The weights of~$P'$ are the same as those of~$P$, so~$P'$ is still \KE. However, the automorphism group of~$P'$ is now only generated by a reflection. Thus,~$P'$ is a \KE\ triangle which is not symmetric.
\end{example}
\begin{remark}
Note that a similar phenomenon was discussed in~\cite[Example~1.7]{on_KEs}. They looked at a triangle~$P$ with~$X_P = \PP^2 / \ZZ_{11}$. In this case,~$P$ had a trivial automorphism group.
\end{remark}
We want to find more examples of non-symmetric \KE\ Fano polygons. To do this, we try to extend what we did in Example~\ref{eg:nonsym_KE_tri} to other symmetric polygons. In particular, we consider other symmetric polygons with respect to finer lattices in an attempt to destroy the symmetry but keep the \KE\ property. Of course, if a polygon is centrally symmetric with respect to one lattice, it will be centrally symmetric with respect to any other lattice. Thus, we restrict our attention to~$3$-symmetric polygons.
\begin{proposition}\label{prop:KE_nonsym_from_3sym}
Let~$h,k$ be integers such that~$h,k \ge 2$ and~$k$ is coprime to~$h$,~$h-1$, and~$2h-1$. Then the polygon
\[
P_{h,k} \coloneqq \conv\begin{pmatrix}
 -hk & (1-h)k & (2h-1)k & (2h-1)k & (1-h)k & -hk \\
 1-h & -h & -h & 1-h & 2h-1 & 2h-1
\end{pmatrix}
\]
is a \KE\ Fano polygon which is not symmetric. Moreover,~$P_{h,k}$ is of type~$B_{\infty}$. If we further suppose that~$k \ge 2h-1$, then~$P_{h,k}$ is not minimal.
\end{proposition}
\begin{proof}
We first show that the polygon is \KE. The polygon~$P_{h,1}$ is~$3$-symmetric, and thus \KE. The polygon~$P_{h,k}$ is isomorphic to~$P_{h,1}$ with respect to the lattice~$N + (1/k, 0)\ZZ$. The weights of a polygon are the same, regardless of which lattice it is considered on. In particular,~$P_{h,k}$ has the same weights are~$P_{h,1}$. Thus,~$P_{h,k}$ is also \KE.

Next, given the assumption that~$k$ is coprime to~$h$,~$h-1$, and~$2h-1$, it easily follows that the vertices of~$P_{h,k}$ are primitive. Thus, since the origin is still contained in the strict interior, the polygon is Fano. It remains to show that~$P_{h,k}$ is not symmetric. To do this, we compare the lengths of its edges. Starting from the bottom-left edge of~$P_{h,k}$ and going anticlockwise, the lengths are~$1$,~$(3h-2)k$,~$1$,~$3h-2$,~$k$, and~$3h-2$. Since~$k \ge 2$, the lengths don't repeat three times. Thus,~$P_{h,k}$ is not~$3$-symmetric. Since it's clearly not centrally symmetric, we may conclude that the polygon isn't symmetric.

Now, we determine the type of~$P_{h,k}$. We compute its barycentric transformation and obtain that~$B(P_{h,k})$ has vertices~$\pm(k,-2)$,~$\pm(2k,-1)$, and~$\pm(k,1)$. In particular,~$B(P_{h,k})$ is centrally symmetric. Thus, by~\cite[Theorem~1.6]{barycentric_transformations},~$P_{h,k}$ is of type~$B_{\infty}$.

Finally, we show that if~$k \ge 2h-1$, then the polygon is not minimal. The bottom edge of~$P_{h,k}$ has length~$(3h-2)k$ and height~$h$, so it is a long edge. The top edge has length~$k$ and height~$2h-1$. Since~$k \ge 2h-1$, by assumption, it follows that it is also a long edge. Thus,~$P_{h,k}$ is not minimal.
\end{proof}
So, we have found an infinite number of non-symmetric, non-triangle \KE\ Fano polygons.
\subsection*{Acknowledgements}
I would like to thank my supervisors Alexander Kasprzyk and Johannes Hofscheier for their patient guidance and for the many helpful discussions we had throughout this project.
\bibliographystyle{plain}
\bibliography{References}
\end{document}